\documentclass{article}

\usepackage[utf8]{inputenc}
\usepackage[T1]{fontenc}
\usepackage[english]{babel}
\usepackage{amsmath}
\usepackage{amsthm}
\usepackage{amsfonts}
\usepackage{amssymb}  
\usepackage{xcolor}
\usepackage[a4paper,left=3cm,right=3cm]{geometry}
\usepackage{graphicx}
\usepackage{csquotes}
\usepackage{hyperref}

\usepackage[backend=biber,style=numeric]{biblatex}
\DeclareBibliographyDriver{article}{%
  \printnames{author}%
  \iffieldundef{title}{}{\addcomma\addspace\mkbibemph{\printfield{title}}}%
  \iffieldundef{journaltitle}{%
    \iffieldundef{number}{}{\addcomma\addspace\printfield{number}}%
  }{%
    \addcomma\addspace\printfield{journaltitle}%
    \iffieldundef{year}{}{\addcomma\addspace\printfield{year}}%
  }%
  \iffieldundef{pages}{}{\addcomma\addspace\printfield{pages}}%
  \finentry}

\DeclareBibliographyDriver{book}{%
  \printnames{author}%
  \iffieldundef{title}{}{\addcomma\addspace\mkbibemph{\printfield{title}}}%
  \iffieldundef{publisher}{}{\addcomma\addspace\printfield{publisher}}%
  \iffieldundef{year}{}{\addcomma\addspace\printfield{year}}%
  \iffieldundef{pages}{}{\addcomma\addspace\printfield{pages}}%
  \finentry}

\theoremstyle{definition}

\newcounter{sharedcounter}[subsection]

\newtheorem{lemma}[sharedcounter]{Lemma}
\newtheorem{proposition}[sharedcounter]{Proposition}
\newtheorem{definition}[sharedcounter]{Definition}
\newtheorem{remark}[sharedcounter]{Remark}
\newtheorem{corollary}{Corollary}
\newtheorem{theorem}{Theorem}

\DeclareMathOperator{\pic}{\text{Pic}}
\DeclareMathOperator{\ch}{ch}
\DeclareMathOperator{\rk}{rk}

\DeclareMathOperator{\td}{td}
\DeclareMathOperator{\Td}{Td}

\newcommand{\ov}[1]{\overline{#1}}

\newcommand{\eq}[2][1]{
\begin{equation*}
	\addtolength{\fboxsep}{0pt}
	\begin{alignedat}{#1}
	#2
	\end{alignedat}
\end{equation*}
}

\newcommand{\eqtag}[3][1]{
\begin{equation}
   \label{#3}
	\begin{alignedat}{#1}
	#2
	\end{alignedat}
\end{equation}
}

\begin{document}

\title{Chern classes of the multilayer fractional quantum Hall bundle on Riemann surfaces}

\author{Mar\'{i}a ABAD ALDONZA and Florent DUPONT}
\date{}

\maketitle

\vspace{1cm}

\begin{abstract}

The so-called  multilayer wave functions were introduced in the study of the fractional Quantum Hall effect by Halperin and others. They are defined with the help of a symmetric matrix $K$ in $M^k(\mathbb{N})$, which encodes the couplings between the $k$ layers where particles live. We study the multilayer quantum states in the case where each layer is a Riemann surface of genus $g$. These states form a vector bundle over the Jacobian variety of the Riemann surface, or the space of Aharonov-Bohm fluxes in physics terminology. Burban-Klevtsov have determined the rank and first Chern class of this bundle when $g=1$, and Klevtsov-Zvonkine computed the Chern character of this bundle in the single layer case for any genus. We generalize the latter approach to the multilayer case and compute the Chern character of the multilayer bundle for any genus, and any possible number of non-localized quasi-holes, under the assumption that the bilinear form associated to $K-I$ is non negative. The key tools are the Grothendieck-Riemann-Roch formula, Berezin integration and Wick's formula for exterior algebras. We show that when all quasi-holes are localized, the Chern character is compatible with the bundle being projectively flat. Furthermore, for those configurations, the conductance becomes independent of the genus and is equal to the sum of the coefficients of the inverse matrix $K^{-1}$, proving two conjectures by Keski-Vakkuri and Wen \cite{Keski-Vakkuri_Wen_1993}. We also find the relation linking $K$, the genus of the surface, the magnetic field and the number of non-localized quasi holes in each layer, which was studied under the name of "shift formula" in \cite{Fröhlich_Zee_1991,Wen_Zee_1992}. Finally, we study conditions on the matrix $K$ under which states having only localized quasi-holes maximize the total particle number, as well as  the asymptotics for large magnetic fields in this scenario.

\end{abstract}

	\tableofcontents

\section{Introduction}

\subsection{The quantum Hall effect}
Take a thin rectangular conductor plate and connect its edges two by two as in Fig. \ref{fig:sample_classical}. We get an object with two loops, say loop 1 and loop 2 on which we can measure voltages $V_1$ and $V_2$ and electric currents $I_1$ and $I_2$. Suppose that there is a constant magnetic field $B$ orthogonal to the plate. Electrons traveling around loop 1 will experience an additional force $-e v \times B$ with $v$ their velocity and $-e$ their electric charge, which will induce a current along loop 2 linear in $V_1$. This phenomenon is called the classical hall effect, and the factor of proportionality $R_H=\frac{V_{1}}{I_{2}}$ is called the transversal resistance. 

\begin{figure}[h]
    \centering
    \includegraphics[width=0.4\linewidth]{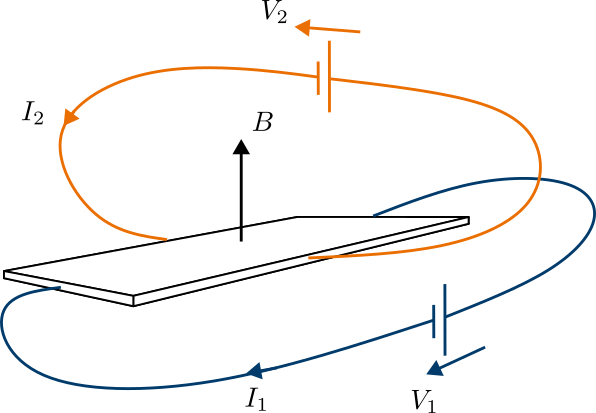}
    \caption{Example of experiment to observe the Hall effect} 
    \label{fig:sample_classical}
\end{figure}

Classical electrodynamics predicts $R_H$ to be a linear function of $B$. Yet, for certain two-dimensional electron systems under strong magnetic fields and low temperatures, $R_H$ exhibits plateaus on which the conductance $\sigma_H:=\frac{1}{R_H}$ is a rational number with a small denominator in units of $\frac{e^2}{h}$.

\begin{figure}[ht]
    \centering

\begin{minipage}{0.4\textwidth}
    \centering
    \includegraphics[width=\textwidth]{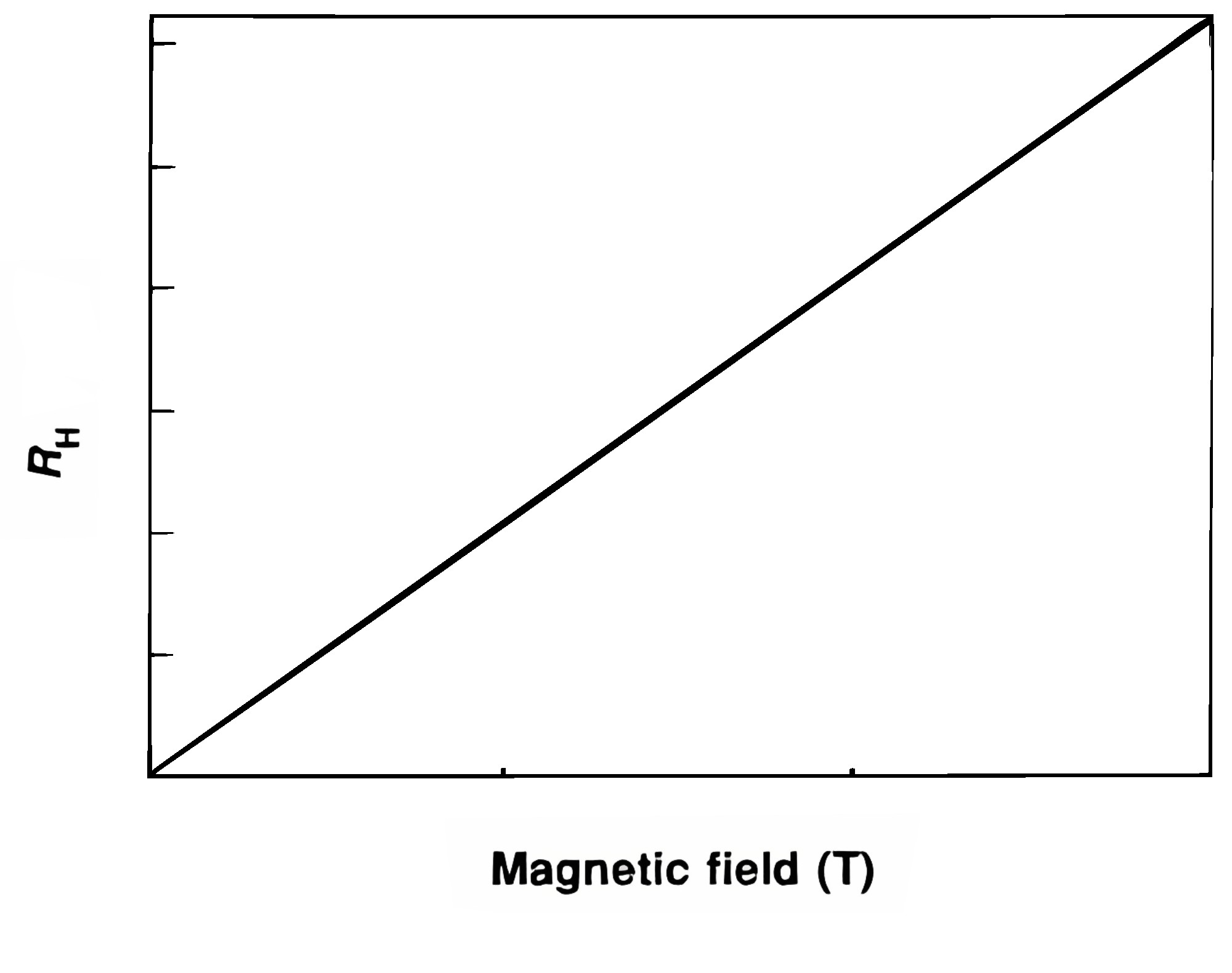}

\end{minipage}
\begin{minipage}{0.4\textwidth}
    \centering
    \includegraphics[width=\textwidth]{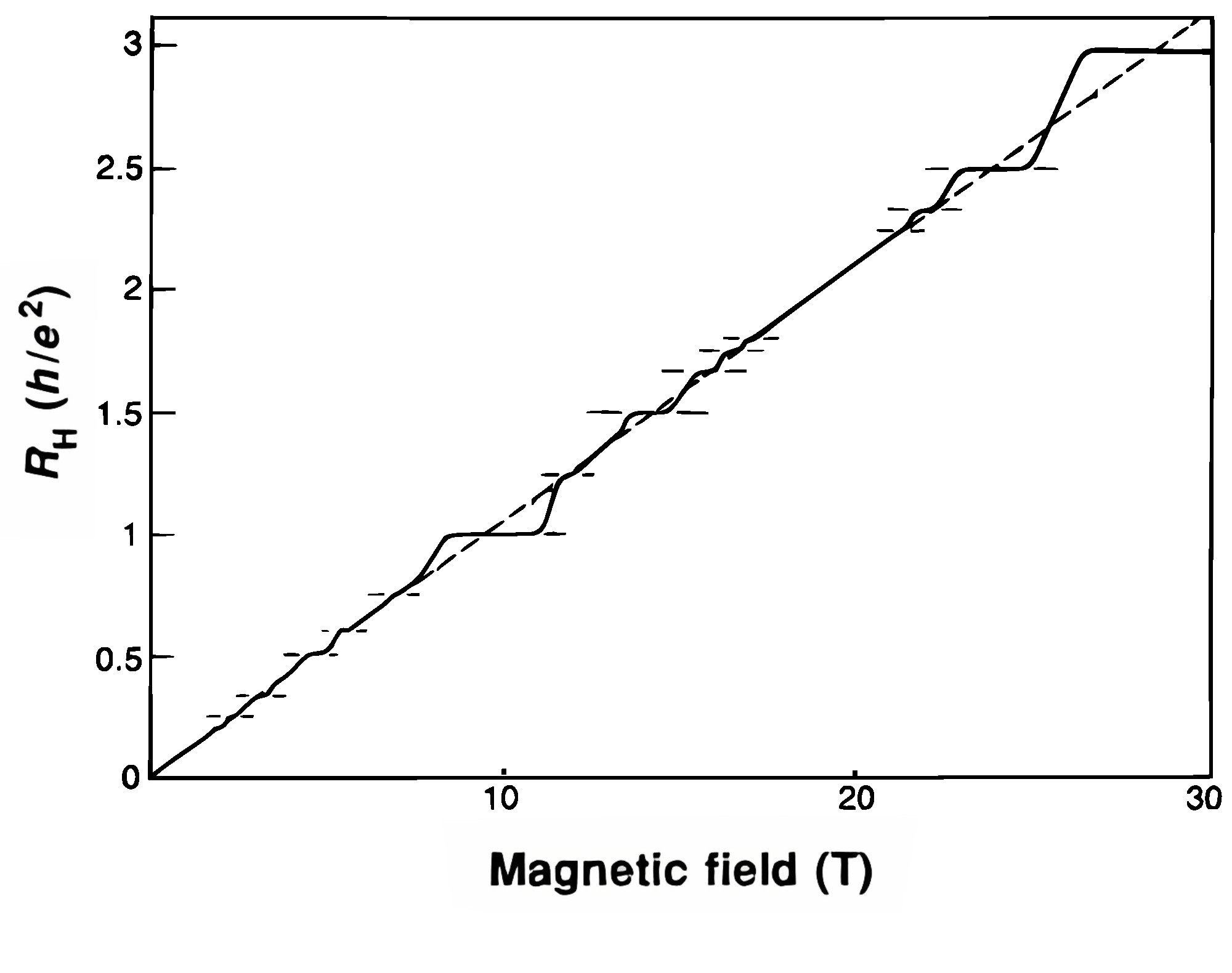}
\end{minipage}

  \caption{Transversal resistivity as predicted by classical mechanics, (left, illustrative image) and as observed on some systems under strong magnetic field (right, image from   \cite{Willett_Eisenstein_Störmer_Tsui_Gossard_English_1987})} 
  
  \label{fig:FQHE_experimental}
\end{figure}

When the experiment is run on a layered conductor such that electrons cannot change layers, we talk about multilayer quantum Hall effect. In such a setting, electrons on two different layer are distinguishable, while electrons on the same layer are not. \\

Quantum Hall effect was first modeled on the plane, but it is customary to consider it on other surfaces as well. This goes back to Laughlin \cite{Laughlin_1981} who considered it on a cylinder for his shift register argument to compute the conductance, Haldane \cite{Haldane_1983} who considered it on a sphere and Thouless, Kohmoto, Nightingale \cite{Thouless_Kohmoto_Nightingale_Den_Nijs_1982} as well as Haldane and Rezayi \cite{Haldane_Rezayi_1985} considered it on a torus. Wen, Niu \cite{Wen_Niu_1990} extended the computation of the degeneracy of the ground state to higher genus surfaces, Avron, Seiler and Zograf \cite{Avron_Seiler_Zograf_1994} generalized Laughlin's argument in the integer case where $g>1$. More recently  Klevtsov, Ma, Marinescu and Wiegmann \cite{Klevtsov_Ma_Marinescu_Wiegmann_2017} gave an asymptotic expansion for the partition function for the integer QHE, from which the conductance can be recovered and Klevtsov \cite{Klevtsov_2016,Klevtsov_2019} constructed an explicit basis of Laughlin states on higher genus surfaces.

In \cite{Klein_Seiler_1990} Klein and Seiler studied the quantum hall effect with interactions in genus $1$ and identified the conductance with the first Chern class divided by rank of a vector bundle. This is why, in this article, the conductance is identified with $\frac{c_1(V)}{\rk(V)}$ with $V$ the vector bundle of ground-states.

\subsection{Laughlin wavefunctions and their generalization to multilayer wavefunctions}

The study of the integer quantum Hall effect (when $\sigma$ is an integer in units of $\frac{e^2}{h}$) was carried out by Laughlin in \cite{Laughlin_1981} and corresponds to the case where the charged particles are non-interacting.

It is thought that the fractional quantum Hall effect is due to strong interactions among the charged particles. However, the analytical study of strongly coupled electron dynamics is too complicated, since it is not known how to solve the Schrödinger equation when the Hamiltonian presents coulomb interactions. In \cite{Laughlin_1983}, Laughlin focused on the case when $\sigma=\frac{1}{b}$ for $b$ odd and proposed an ansatz, or educated guess, for the ground state of an $N$-particle system over $\mathbb{C}$.

$$\psi_b(z_1,...,z_N)=\prod_{i<j}(z_i-z_j)^b \cdot e^{-\frac{1}{4}\sum_i|z_i|^2}$$ 

Halperin, and latter in more generality Keski-Vakkuri and Wen, proposed the following "Laughlin type" ansatz for the study of multilayer quantum Hall states over $\mathbb{C}$ \cite{Halperin_1983,Halperin_1984,Keski-Vakkuri_Wen_1993}:

\begin{equation}\label{eq:Halperin_wavefunctions}
    \psi_K(z^i_\lambda)=\prod_{i\neq  j,\lambda,\mu}(z_\lambda^i-z_\mu^j)^{K_{ij}}\prod_{i,\lambda<\mu}(z_\lambda^i-z_\mu^i)^{K_{ii}}\cdot e^{-\frac{1}{4}\sum_{\lambda, i}|z_\lambda^i|^2}
\end{equation}
where $K$ is an integer-valued matrix with non-negative coefficients that controls the coupling between layers, $i,j$ are labeling layers, and $\lambda,\mu$ are labeling particles among layers. Recently, these wavefunctions where conjectured to appear in multilayer graphene lattices, see \cite{Tarnopolsky_Kruchkov_Vishwanath_2019, Liu_Hao_Watanabe_Taniguchi_Halperin_Kim_2019}.\\

The multilayer torus case was first studied in \cite{Keski-Vakkuri_Wen_1993,Varnhagen_1995}. In the first paper, Keski-Vakkuri and Wen also made two conjectures about quantum Hall states on a Riemann surface of genus $g$. They conjectured that for maximally filled multilayer states, called incompressible multilayer states, the degeneracy of the lowest energy level is $\det(K)^g$. Their second conjecture is that the conductance of the system is independent of the genus and equal to $|K^{-1}|:=\sum_{ij} K^{-1}_{ij}$, the sum of the coefficients of the inverse matrix.\\

For the Laughlin, or single-layer states,  Klevtsov and Zvonkine  \cite{Klevtsov_Zvonkine_2025}  recently constructed a vector bundle $V$ over $\pic^dC$ whose fibers correspond to the vector spaces of Laughlin wavefunctions over a genus $g$ Riemann surface.  They obtained the rank of this bundle $V$ and all of its Chern classes by computing its Chern character.
In \cite{Burban_Klevtsov_2025} Klevtsov and Burban generalized the algebro-geometric constructive approach to the multilayer case on a torus. In \cite{Burban_Klevtsov_2024} they showed that the vector bundle constructed is projectively flat.\\

Our goal in this paper is to study wavefunctions that are the generalization of 
Halperin, Keski-Vakkuri and Wen's ansatz to the geometrical setting where electrons live on a multilayer compact Riemann surface of any genus. We apply the algebro-geometric approach pioneered by Klevtsov, Zvonkine and Burban to this case. As a result, we prove both conjectures formulated by Keski-Vakkuri and Wen about the degeneracy and conductance of the system in theorem \ref{thm:main}. We also obtain the relation between $K,d,g$, $\vec{n}$, which was studied in \cite{Fröhlich_Zee_1991,Wen_Zee_1992} ("shift formula").

\subsection{Geometric model}

\label{subsec:geometricmodel}

Let $C$ be a closed orientable genus $g$ surface with a Riemanian metric. In particular, $C$ has the structure of a complex algebraic curve. In this article, $C$ will represent the conductive surface on which the particles are constrained. On this surface, we assume the presence of a magnetic field everywhere orthogonal to $C$, and constant over time. We denote by $\gamma_1\dots,\gamma_{2g}$ the $2g$-cycles of $C$ up to homotopy. $C$ has no boundary on which we can impose voltage differences as in the case of Figure \ref{fig:sample_classical}, but we can induce an electric field circulation along $\gamma_i$ by varying over time the magnetic flux $\phi_i$ going through it. The magnetic flux time derivatives $\frac{d\phi_i}{dt}$ thus play the role of the voltages $V_i$. This hints at the interest of studying the space of possible magnetic field configuration evolutions on $C$, which we do in the section to come. 

\begin{figure}[h]
    \centering
    \includegraphics[width=\linewidth]{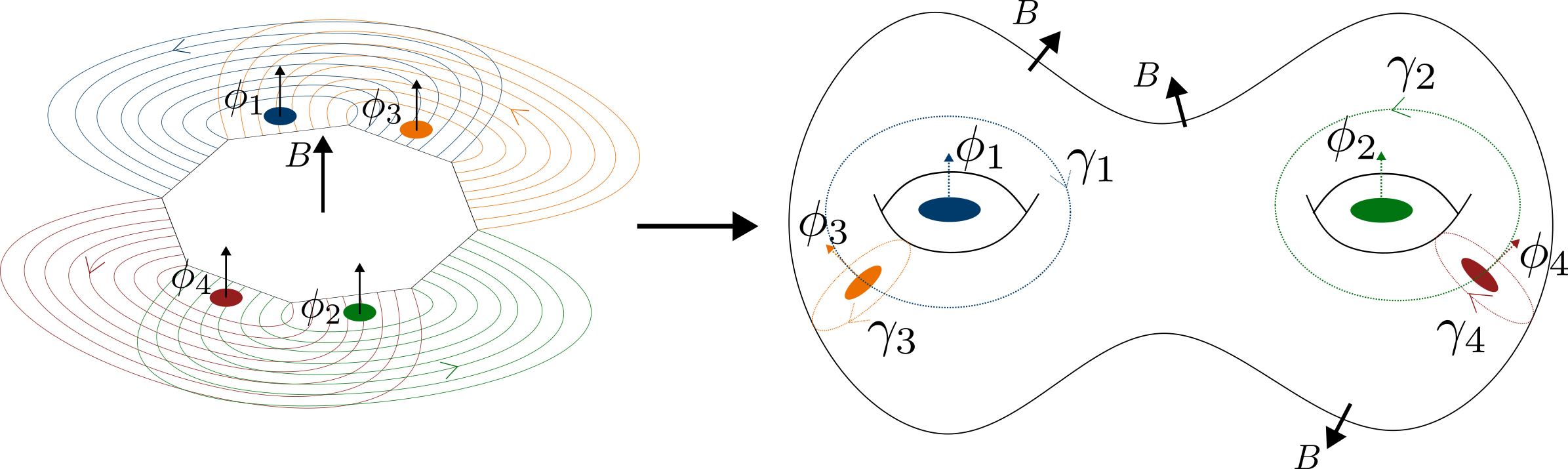}
    \caption{Illustration of a conductive sample in our model. By identifying the edges connected by cables of the same color, one obtains a genus two surface $C$. During the experiment, the magnetic field $B$ is kept fixed, while the magnetic fluxes $\phi_i$ going trough the cycles $\gamma_i$ of $C$ can change adiabatically trough time. They induce currents on $C$ which enable the study of the quantum Hall effect.}  
    \label{fig:conductor-as-Riemann-surface}
\end{figure}

Note that to measure an electric conductance in this setting, one needs to specify two homology classes: the class of one cycle along which the current is measured, and another through which the magnetic flux is varied along time. This is the intuitive reason why the conductance can be seen as an object which takes homotopy classes as an input and outputs a number, i.e. a cohomology class.

\paragraph{The space of magnetic field configurations on the surface $C$\\}

An electromagnetic field $(E,B)$ on the surface $C$ is the data of a principal $U(1)$-bundle $\mathbf{P}\to C\times \mathbb{R}$ equipped with a connection $\nabla$ up to gauge equivalence. A gauge equivalence is a continuous map $\phi:C\times \mathbb{R}\rightarrow U(1)$, it defines a fiber bundle automophism of $\mathbf{P}$ and transforms the connection by addition of a 1-form $\phi^{-1}d\phi$ over $C\times \mathbb{R}$. Note that the degree $d=\deg(P_{|C\times \lbrace t \rbrace})$ is independent of time $t$. We can construct a local trivialization of $\mathbf{P}$ from a local trivialization of $P=\mathbf{P}_{|C\times \lbrace 0 \rbrace }$ by using the parallel transport along the curves $\lbrace z \rbrace \times \mathbb{R}$. In those trivializations, $\nabla=d+\omega_t$ where $\omega_t$ is a time dependent 1-form on $C$. The electric field $E$ and the magnetic field $B$ can be recovered from the connection by

$$E=\frac{\partial \omega_t}{\partial t} \qquad \qquad B = \star d \omega_t $$
where $\star$ is the hodge star operator.

We can now identify $\mathbf{P}\simeq P\times \mathbb{R}$ and $\nabla$ with a collection $\nabla_t$ of connections on $P$.  Given any $\nabla_{t_1}$ and $\nabla_{t_2}$, their difference is a 1-form corresponding to $\int_{t_1}^{t_2}E \ dt$. The gauge transformations compatible with the identification $\mathbf{P}\simeq P\times \mathbb{R}$ (where connections on $\mathbf{P}$ don't have a $dt$ term) are constant along $\mathbb{R}$, hence defined by continuous functions $\phi:C \to U(1)$. We would like to consider the space of magnetic field configurations over $C$. Paths inside this space correspond to time evolutions of the system.

First note that, as stated before, we assume the magnetic field $B$ to be constant over time on the surface. This means all connections $\nabla_t$ have the same curvature. We also take $E$ to be divergence free, which amounts to considering the surface is a good electric conductor without charge accumulation. Together, those two hypothesis imply that $E=\frac{\partial \omega_t}{\partial t}$ is both closed and divergence free, hence an harmonic 1-form. Since the space of harmonic 1-forms is canonically isomorphic to $H^1(C,\mathbb{R})$, the space of connections with fixed curvature up to gauge transformation $\phi: C \to \mathbb{R} / \mathbb{Z}$ that can be lifted to $\mathbb{R}$ is an affine space $W$ with underlying vector space $H^1(C,\mathbb{R})$. This is not yet the parameter space we were looking for, as we have only formed equivalence classes for gauge transformations that can be lifted to $\phi: C \to \mathbb{R}$. Indeed, for such a type of gauge transformation we have $\phi^{-1} d\phi=0$ in $H^1(C,\mathbb{R})$, \\

Any gauge equivalence $\phi:C \to U(1)$ between $(P,\nabla')$ and $(P,\nabla)$ relates both connections through $\nabla'=\nabla+\phi^{-1} d\phi$. This $\phi^{-1} d\phi$ is a 1-form that corresponds to an integer cohomology class. Indeed, locally $\phi^{-1} d\phi=d \log \phi$ thus integrating this form on a cycle $\gamma$ gives the index of the loop $\phi \circ \gamma$. Conversely, given an integer cohomology class $\alpha$, we can construct a gauge transformation $\phi(z)=\exp(2\pi i\int_{z_0}^z \alpha )$ that is well defined precisely because integrating $\alpha$ along loops gives an integer number.  Therefore, the space of magnetic field configurations is $W$ quotiented by the action of $H^1(C,\mathbb{Z})$.

\paragraph{Quantum mechanics of a single particles on $C$ and the correspondence between magnetic field configurations and holomorphic line bundles.\\}

A wavefunction for a single charged particle evolving on $C$ in the magnetic field defined by $P$ is a smooth section of the complex line bundle $L=P \times_{U(1)} \mathbb{C}$. $L$ is equipped with the connection from $P$ and a Hermitian metric inherited from $\mathbb{C}$. The Hamiltonian can be written $$H=-\nabla^\dag \nabla +a B + b R$$ where $R$ is the curvature of the surface $C$ and $a$, $b$ are material dependant constants. 

$H$ has a discrete spectrum bounded below as a self adjoint elliptic operator on a compact surface. Under the condition that $(a-1) B + b R$ is constant over $C$, the lowest eigenspace can be written as the space of holomorphic sections of $L$ for some particular holomorphic structure obtained from the connection through the canonical isomorphism $W/ H^1(C,\mathbb{Z}) \overset{\sim}{\to}  \pic^d(C)$.\\

To see how this map is constructed, first recall that since $C$ has a complex structure, the connection $\nabla$ splits as a sum of a holomorphic and antiholomorphic component: $\nabla=\nabla^{1,0} + \nabla^{0,1}$. The Bochner-Kodaira -Nakano relation $\nabla^\dag \nabla=2 {\nabla^{0,1}}^\dag \nabla^{0,1} + B$ gives 

$$H=-2 {\nabla^{0,1}}^\dag \nabla^{0,1}  +(a-1) B + b R$$

The holomorphic structure on $L$ is then defined by $\nabla^{0,1}=\ov{\partial}_L$ where the right hand side is the Dolbeault operator associated with the holomorphic structure on $L$. For this holomorphic structure, the lowest eigenspace of $H$ is $\ker{\ov{\partial}_L}=H^0(C,L)$.

\paragraph{Multilayer wavefunctions as sections of a line bundle\\}

Consider a conductor with $k$ layers indexed by $i$, each of them identified with the compact Riemann surface $C$. We assume that charged particles in the conductor are confined to their departing layer and cannot switch to another one. We want to generalize the wavefunctions  (\ref{eq:Halperin_wavefunctions}) to this geometrical setting. 

\begin{definition}
    Given $k\in \mathbb{N}_{>0}$, we call a \textit{configuration of a $k$-layered fractional quantum Hall system} the data of a quadruple $(K,g,d,\vec{n})$ where:
    \begin{itemize}
        \item[] $K\in\mathcal{M}^k(\mathbb{N})$ is a symmetric matrix whose coefficients give the couplings between the $k$ layers. 
        \item[] $g\in \mathbb{N}$ is the genus of the surface $C$.
        \item[] $d\in \mathbb{N}$ is the flux of the magnetic field going trough $C$.
        \item[] $\vec{n}\in \mathbb{N}^k$ is a column vector such that $n_i>2g-1 \quad \forall i$. It specifies the number of particles in each layer.  
    \end{itemize}

\end{definition}

\begin{definition}
    Given a configuration of a $k$-layered fractional quantum Hall system $(K,g,d,\vec{n})$, we can associate a vector
    $$\vec{p}=\vec{d}-K\vec{n}-(g-1)\vec{K}$$
    where $\vec{d}$ denotes the column vector of size $k$ whose entries are all $d$ and $\vec{K}=(K_{11}=,\dots,K_{kk})^T$ the column vector whose entries are the diagonal entries of $K$. Each entry $p_i$ of $\vec{p}$ is called the \textit{number of non-localized quasi-holes} in layer $i$. We will give an insight into this name choice later.
\end{definition}

Let $(K,g,d,\vec{n})$ be a configuration of a $k$-layered fractional quantum Hall system and $L\rightarrow C$ a degree $d$ holomorphic line bundle ($L$ encodes a particular magnetic field configuration). Denote $N=\sum n_i$ the total number of charged particles. We can build a line bundle over the cartesian product $C^N$ by $L^{\boxtimes N}=\bigotimes_j \pi_j^*L$, where $\pi_j:C^N\rightarrow C$ denotes the projection over the $j$-th factor. We will call multilayer quantum Hall states associated to $(K,g,d,\vec{n})$ the sections of $L^{\boxtimes N}$ that verify the same hypothesis Laughlin and his followers used to build their ansatzs for the ground states of the system:

\begin{itemize}
    \item Their restriction to each copy of $C$ gives a holomorphic section of $L$.
    \item They are symmetric (resp. antisymetric) under particle exchange of two particles inside the same layer $i$ when $K_{ii}$ is even (resp. odd). 
    \item They vanish at order $K_{ij}$ when two particles in layers $i$ and $j$ respectively are in the same position.
\end{itemize}
Note that these hypothesis ensure that what we have called multilayer wavefunctions behave locally as holomorphic functions having a factor $\prod_{i\neq  j,\lambda,\mu}(z_\lambda^i-z_\mu^j)^{K_{ij}}\prod_{i,\lambda<\mu}(z_\lambda^i-z_\mu^i)^{K_{ii}}$. \\

Now, let $\Delta_K$ represent the divisor on $C^{{\scriptscriptstyle \sum_i} n_i}$ defined as

$$\Delta_K= \sum_{i\leq j}K_{ij}\Delta_{ij}$$
where $\Delta_{ij}:=\bigcup\{z^i_\lambda=z^j_\mu\}$ for $\lambda=1...n_i$ and $\mu=1...n_j$. Note that if $i=j$, we additionally ask $\lambda< \mu$ since otherwise $\Delta_{ii}$ would be the whole layer $n_i$.\\

Holomorphic sections of $L^{\boxtimes N}(-\Delta_K):=L^{\boxtimes N} \otimes O(-\Delta_K)$ correspond to holomorphic sections of $L^{\boxtimes N}$ that vanish on $\Delta_{ij}$ to the order $K_{ij}$. If we quotient $C^{n_i}$ (the $n_i$ copies of $C$ corresponding to particles in layer $i$) by the symmetric group $S_{n_i}$ for all $i$, we can build a line bundle $L_K$ over $\prod_i S^{n_i} C$ whose pullback on $C^N$ is exactly $L^{\boxtimes N}(-\Delta_K)$. Any holomorphic section of $L_K$ corresponds to a multilayer wavefunctions (for the magnetic field configuration given by $L$). \\

\paragraph{Construction of the Multilayer wavefunction bundle $V_{K,d,g,\vec{n}}$ over $\pic^d(C)$\\}

Recall that $\pic^dC$ is the space of magnetic field configurations over $C$. We now turn to the problem of building a sheaf over $\pic^d(C)$ whose stalks are exactly the vector spaces of multilayer wavefunctions for the given holomorphic structure on $L$. We will later prove that this sheaf is a vector bundle and thus call it the Multilayer wavefunction bundle, denoted by $V_{K,d,g,\vec{n}}$. This is the vector bundle whose first Chern class divided by rank is identified with the electric conductance.\\

To construct this sheaf, we will construct a universal line bundle above $\prod_i S^{n_i}C \times \pic^d(C)$ and push it forward on $\pic^d(C)$. First, recall there exist a line bundle $P_0$ on $C \times \pic^0(C)$ such that $$\forall \xi \in \pic^0(C) \quad {P_0}_{|C\times \lbrace \xi \rbrace} \simeq \xi.$$

Such a line bundle is called a Poincaré line bundle because it comes from a Poincaré line bundle on $\pic^0(C) \times \hat{\pic^0}(C)$ via pullback by the Abel-Jacobi embedding and the identification of $\pic^0(C)$ with its dual using the canonical polarization of $\pic^0(C)$. Now fix a point $z_0 \in C$. The condition that ${P_0}_{| \lbrace z_0 \rbrace \times \pic^0(C)}$ be trivial fixes $P_0$ uniquely. Note that changing the choice of point $z_0\in C$ here results in a Poincaré bundle that only differs from $P_0$ by a degree $0$ line bundle over $\pic^0(C)$. Now, the choice of $z_0$ also allows us to identify $\pic^d(C)$ with $\pic^0(C)$, where the isomorphism $u:\pic^d(C) \to \pic^0(C)$ is given by multiplication by $O(-d z_0)$. This choice of point thus gives a universal line bundle of degree $d$ line bundles $$P_d=p_1^\star O(d z_0) \otimes u^\star P_0$$ with $p_1:C\times \pic^d(C) \to C$.

We now do the same construction as previously, replacing the line bundle $L\to C$ by $P_d \to C \times \pic^d(C)$. Instead of obtaining $L_K \to \prod_i S^{n_i} C $, we obtain a line bundle $\mathcal{L}_K$ over $\prod_i S^{n_i} C \times \pic^d(C)$ whose restriction to $\prod_i S^{n_i} C \times \{[L]\}$ for any given degree $d$ line bundle $L$ over $C$ is isomorphic to $L_K$. \\

Finally, let $p:\prod_i S^{n_i}C \times \pic^d(C)\to \pic^d(C)$ be the projection onto the second factor. We will call \textit{Multilayer wavefunction  bundle} the direct image $p_*\mathcal{L}_K$.
The fibers above $L$ of this bundle correspond as desired to $H^0\left(\prod S^{n_i}C, L_K\right)$, the vector spaces of holomorphic sections of $L_K$.

\subsection{Main result}

In the statement that follows, we let $V_{K,g,d,\vec{n}}=p_\star \mathcal{L}_K$ where $\mathcal{L}_K$ is the universal line bundle corresponding the the configuration $(K,g,d,\vec{n})$. We also denote by $\theta$ the class in $H^2(X,\mathbb{Z})$ of a theta divisor on $\pic^d(C)$.

\begin{theorem}
\label{thm:main}
     Let $(K,d,g,\vec{n}_0)$ be a configuration such that $K-I$ corresponds to a non-negative bilinear form and $$ K\vec{n}_0=\vec{d}-(g-1) \vec{K}.$$ 
     The Chern character of $V_{K,g,d,\vec{n}_0}$ is $$\ch(V_{K,g,d,\vec{n}_0})=\det(K)^g e^{-|K^{-1}| \theta}$$
    
    where $|K^{-1}|$ is the sum of the coefficients of $K^{-1}$.
\end{theorem}

\begin{corollary}
Let $(K,d,g,\vec{n}_0)$ be a configuration as in theorem \ref{thm:main}. The degeneracy of the associated ground state of the system and its conductance are respectively $$\rk(V_{K,g,d,\vec{n}_0})=\ch_0(V_{K,g,d,\vec{n}_0})=\det(K)^g$$ and $$\sigma=\frac{\ch_1(V_{K,g,d,\vec{n}})}{\ch_0(V_{K,g,d,\vec{n}})}=-|K^{-1}|\theta $$
\end{corollary}

Theorem \ref{thm:main}
 proves Keski-Vakkuri and Wen's conjecture about the ground state 
degeneracy and conductance of a multilayer quantum system. This is also compatible with the vector bundle $V$ being projectively flat in the sense that it passes the test proposed in \cite{geometric_test}.\\

In this article, we also compute the Chern character $\ch\left(V_{K,g,d,\vec{n}}\right)$ for any $\vec{n}$ (see theorem \ref{thm:multilayer_Chern_character}). This case corresponds to not completely filled states, also called states with non-localized quasi-holes or compressible 
states. We compute their associated degeneracy and conductance.\\

\begin{remark}
From a physics perspective, given a setup like the one of Fig. \ref{fig:conductor-as-Riemann-surface}, a representative of the class $\theta$ is the two form $\sum_{i=1}^g d\phi_i \wedge d\phi_{i+g}$ where $\phi_i$ is the flux going trough the cycle $\gamma_i$ and $g=2$, and thus $\sigma=-|K^{-1}| \sum_i d\phi_i \wedge d\phi_{i+g}$. See \cite{Avron_Seiler_Zograf_1994} for more details on the description of the conductivity as a differential form in the case of the integer quantum Hall effect on a genus $g$ surface.
\end{remark}

\begin{theorem}
\label{thm:maximalparticules}

 Let $(K,d,g,\vec{n})$ be a configuration such that the vector $\vec{p}=\vec{d}-K\vec{n}-(g-1) \vec{K}$ has at least one negative component. Then $\rk(V_{K,g,d,\vec{n}})=0$.

 In particular, suppose the configuration $(K,d,g,\vec{n_0})$ is such that $K-I$ is non negative and $K\vec{n}_0=\vec{d}-(g-1) \vec{K}.$ Then this $\vec{n_0}$ has the property that
 $$\rk(V_{K,g,d,\vec{n}_0})>0$$ and 
$$\rk(V_{K,g,d,\vec{n}})=0$$
for any $\vec{n}>\vec{n}_0$ in the sense that $\vec{n}-\vec{n}_0$ has non negative components and at least one positive component.

In the case where $C_i:=\sum_{j} K_{ji}^{-1} \geq 0$ for any $i$, $\vec{n}_0$ also maximizes the total particle number $N=\sum_i (n_0)_i$. 
    
\end{theorem}

The relation $K\vec{n}_0=\vec{d}-(g-1) \vec{K}$, which we obtain for states without non-localized quasi-holes, also called incompressible states, was studied in \cite{Fröhlich_Zee_1991,Wen_Zee_1992}

\begin{remark}
    Let $(K,d,g,\vec{n})$ be a configuration with localized quasi-holes. The location of the quasi-holes in layer $i$ is given by an effective divisor $D_i$. We can build a bundle of multilayer wavefunctions with localized quasi-holes\ $V_{K,g,d,\vec{n}, \lbrace D_i \rbrace_i}$ over $\pic^d(C)$ in the same way that we built $V_{K,g,d,\vec{n}}$. We do so by replacing $O(dz_0)$ by $O(dz_0-D_i)$ in each copy of $P_d$, depending on the layer.
    Then Theorem \ref{thm:main} and Theorem \ref{thm:maximalparticules} are still valid if we replace $V_{K,g,d,\vec{n}}$ by $V_{K,g,d,\vec{n}, \lbrace D_i \rbrace_i}$ and $\vec{d}$ by $\vec{d'}=(d-\deg(D_i))_i$, including in the definition of $\vec{p}$.
\end{remark}

\subsection{Outline of the paper}

Recall that $p:\prod_i S^{n_i}C \times \pic^d(C)\to \pic^d(C)$ is the projection on the last factor. The Grothendieck–Riemann–Roch theorem applied to this map and the line bundle $\mathcal{L}_K$ gives:

$$\ch\left(\sum_i (-1)^iR^i p_\star \mathcal{L}_K\right)=p_\star \left(e^{c_1(\mathcal{L}_K)} \td\left(\prod_{i=1}^k S^{n_i}C\right)\right)$$

In section \ref{sec:importantcohomologyclasses}, we introduce the cohomology classes used in our computations.
In section \ref{sec:computingthefirstChernclass}, we express the cohomology class $e^{c_1(\mathcal{L}_K)}\td\left(\prod_{i=1}^k S^{n_i}C\right)$ in terms of the previously introduced ones.
We give a sufficient condition such that all higher direct images of $\mathcal{L}_K$ vanish in the first part of section \ref{sec:applyinggrr}. 
Finally, we carry out the pushforward computation of the Grothendieck-Riemann-Roch formula in the second part of section \ref{sec:applyinggrr}. In section \ref{sec:asymptotics}, we study the large magnetic field asymptotics.

\section{Important cohomology classes}
\label{sec:importantcohomologyclasses}

We are interested in expressing $e^{c_1(\mathcal{L}_K)} \td\left(\prod_{i=1}^k S^{n_i}C\right)$ in terms of a suitable set of cohomology classes. Note that all cohomology classes in our formula are expressed in terms of first Chern classes. The first Chern class of a given vector bundle over a smooth complex variety $Y$ always falls inside its Neron-Severi group $NS(Y)=H^2(Y;\mathbb{Z})\cap H^{1,1}(Y;\mathbb{C})$. 
In this section we are going to introduce elements of the Neron-Severi group of $\prod_{i\in I} S^{n_i}C \times \pic^dC$ to express the aforementioned cohomology class.

\subsection{The Neron-Severi group of the symmetric product $S^N C$} \label{NS symmetric product}

Let $C$ be a genus $g$ complex curve and $S^NC$ its $N$-th symmetric power. We will refer to the subring of $H^\bullet(S^N C;\mathbb{Z})$ spanned by the fundamental classes of algebraic cycles on $S^N C$ through Poincaré duality as "the algebraic part" of $H^\bullet(S^N C;\mathbb{Z})$. Because of the Lefschetz theorem for $(1,1)$-classes (a particular case of the Hodge conjecture), the Neron-Severi group $NS(S^NC)$ coincides with the algebraic part of $H^2(S^NC, \mathbb{Z})$. \\

There is a natural map $\rho_N: S^NC\longrightarrow \pic^NC$ sending points in $S^NC$ to the class of the corresponding degree $N$ divisor in $\pic^NC$.
The fibers of this map are projective spaces (the linear systems of divisors). For $N\geq 2g-1$, Mattuck \cite{Mattuck_2} showed that this map actually forms a projective bundle. Given a $z_0\in C$, the map $\pic^0C \rightarrow \pic^NC$ coming from a translation by $Nz_0$  is an isomorphism. Define a divisor $\xi$ in $S^NC$ as the set of points of $S^N C$ containing $z_0$ at least once. We will denote $\xi$ both the divisor and its image in cohomology through Poincaré duality. Using the same $z_0$, Mattuck \cite{Mattuck} showed that one can fix (up to isomorphism) a vector bundle $E$ of rank $r=N-g+1$ over $J(C)=\pic^0(C)$ such that $S^NC$ is isomorphic to $\mathbb{P}(E)$, the projectivization of $E$, and such that $\xi=c_1(\mathcal{O}_{\mathbb{P}(E)}(1))$ in the Chow group. By the projective bundle formula,

$$\mathcal{A}^*(S^NC)=\mathcal{A}^*(J(C))[\xi]$$
with the relation
$$\xi^r+\rho_N^*c_1(E)\xi^{r-1}+...+\rho_N^*c_r(E)=0$$

Let $\Theta$ be the theta divisor in $\pic^{g-1}C$. Given the previous choice of $z_0$ and the associated translation isomorphisms, we can recover a copy of $\Theta$ in each fixed degree Picard group. In particular, we can consider the divisor class of $\Theta$ inside $\mathcal{A}^*(J(C))$. Let $\theta$ be the Poincaré dual to $\Theta$ in $J(C)$. It follows from Mattuck's computations of $c(E)$ in $\mathcal{A}(J(C))$, present in article \cite{Mattuck}, that $c(E)=e^{-\theta}$. We have thus obtained, for $N\geq 2g-1$, two classes in the Neron-Severi group $NS(S^NC)$. For a general curve, $NS(S^NC)$ is actually spanned by the classes $\theta$ and $\xi$. More information about this topic can be found in \cite{Griffiths_book}.\\

\subsection{Useful elements of the Neron-Severi group of $\prod_i S^{n_i} C \times \pic^d C$}

The notation introduced in this subsection will be used throughout the article.\\

From Künneth theorem we can infer at least two ways of obtaining cohomology classes in $H^2\left( \prod_i S^{n_i} C \times \pic^d C,\mathbb{Z}\right)$. For each factor $y$ in $\prod_i S^{n_i} C   \times \pic^d C$, there is an injection $H^2\left(y, \mathbb{Z} \right)\rightarrow H^2(\prod_i S^{n_i} C \times \pic^d C, \mathbb{Z})$ via the pullback by the corresponding projection. For each pair of factors $(y_1,y_2)$, there is an injection $H^1(y_1, \mathbb{Z})\otimes H^1(y_2, \mathbb{Z})\rightarrow H^2\left(\prod_i S^{n_i} C \times \pic^d C,\mathbb{Z}\right)$ whose image we call "mixed" cohomology classes.\\

For each $S^{n_i} C$, we denote $\theta_i$ and $\xi_i$ the two generators of $NS(S^{n_i}C)$ which we also see by abuse of notation as elements in $NS(\prod_i S^{n_i} C \times \pic^d C)$. Similarly, we denote $\theta$ the theta class in $\pic^d(C)$ and its pullback to $NS(\prod_i S^{n_i} C \times \pic^d C)$. The reader may assume we are using this abuse of notation whenever we define a cohomology class that can be pulled back to $\prod_i S^{n_i} C \times \pic^d C$ using projections and maps of the form $S^N C \to \pic^N C$.\\

The Neron-Severi group of $\prod_i S^{n_i} C \times \pic^d C$ also contains mixed cohomology classes. To see this, take two layers $i,j$ and consider $\sigma_{ij}: \pic^{n_i} C \times \pic^{n_j} C \to \pic^{n_i+n_j} C$ defined by the addition of divisors. We can now define the algebraic class
\eqtag{
\eta_{ij}= \sigma_{ij}^\star \theta_{n_i+n_j}-(\theta_i+\theta_j)
}{eq:mixedclass}
which is a mixed class inside $\pic^{n_i} C \times \pic^{n_j}(C)$. We define in the same way a mixed class $\eta_i$ on $\pic^{n_i} C \times \pic^d(C)$.  
We will soon see that these mixed classes are non-zero.\\

\label{sec:symplectic-basis}
    Denote by $\alpha^1,\cdots,\alpha^g, \beta^1, \cdots,\beta^g$ the basis of harmonic real 1-forms on $C$ dual to a basis of canonical cycles $a^1,\dots,a^g, b^1,\cdots,b^g$ of $H_1(C,\mathbb{Z})$. Consider, for each $i\in\{1,...,k\}$, the Abel-Jacobi map $C \longrightarrow \pic^{n_i}(C)$ given by the same choice of $z_0\in C$ we did in subsection \ref{NS symmetric product}. We denote by $\lbrace \alpha_i^r,\beta_{i}^r, \ 1\leq r\leq g \rbrace$ the symplectic basis of $H^1(\pic^{n_i}(C),\mathbb{Z})$ that pulls back to $\{\alpha^r,\beta^r, \ 1\leq r \leq g\}$ by the corresponding Abel-Jacobi map. We also denote by $\alpha^1,\cdots,\beta^g$ the symplectic basis of $H^1(\pic^{d}(C),\mathbb{Z})$ defined by the same logic. The theta classes of $NS(\prod_i S^{n_i} C \times \pic^dC)$ are expressed, in terms of these basis, as

    \eqtag{
    \theta_i=\sum_{r=1}^g \alpha_i^r \wedge \beta_i^r \ , \qquad \theta=\sum_{r=1}^g \alpha^r\wedge \beta^r}{eq:ThetaAsDifferentialForm}
    
    Since $\lbrace \alpha_i^r+\alpha_j^r, \ \beta_i^r+\beta_j^r, \ 1\leq r\leq g \rbrace$ form a symplectic basis of $\pic^{n_i+n_j}C$, $\sigma_{ij}^\star \theta_{ij}$ we have:
    \eq{
    \sigma_{ij}^\star \theta_{ij}&=\sum_{r=1}^g(\alpha^r_i+\alpha^r_j)\wedge(\beta^r_i+\beta^r_j)\\
    &=\sum_{r=1}^g \alpha^r_i \wedge \beta^r_i + \sum_{r=1}^g \alpha^r_j \wedge \beta^r_j + \sum_{r=1}^g (\alpha_i^r \wedge\beta_j^r +\alpha_j^r \wedge\beta_i^r) }
    from which one can see, recalling Eq. \ref{eq:mixedclass} and Eq. \ref{eq:ThetaAsDifferentialForm} that \eqtag{\eta_{ij}=\sum_{r=1}^g (\alpha_i^r \wedge\beta_j^r +\alpha_j^r \wedge\beta_i^r)}{eq:EtaAsDifferentialForm}
    
    One can show $\eta_{i}=\sum_{r=1}^g (\alpha_i^r\wedge \beta^r +\alpha^r \wedge\beta_i^r)$ with a similar computation.

\section{The first Chern class of $\mathcal{L}_K$ and the Todd class of $\prod_iS^{n_i}C$}
\label{sec:computingthefirstChernclass}

The end goal of this section is to compute $e^{c_1(\mathcal{L}_K)} \td\left(\prod_{i=1}^k S^{n_i}C\right)$ in terms of the cohomology classes introduced in the previous one. 

\subsection{The Todd class of $\prod_iS^{n_i}C$}

Let $\td$ be the formal power series defined by $\td(x)=\frac{x}{1-e^{-x}}$. It can be used to define the Todd class of a line bundle $L$ as $\Td(L)=\td(c_1(L))$.

The Todd class is multiplicative, so $$\Td\left(\prod_iS^{n_i}C\right)=\prod_i\Td\left(S^{n_i}C\right)$$ 
For a derivation of the Todd class of the symmetric product $S^{n_i} C$ for $n_i>2g-1$, we refer to \cite{Klevtsov_Zvonkine_2025} who carried the computation using similar methods to \cite{Mattuck},\cite{MacDonald_1962}.

$$\Td(S^{n_i}C)=\td(\xi_i)^{n_i+1-g} e^{\theta_i \frac{\td \xi_i -\xi_i-1}{\xi_i}}$$ 

In conclusion,
 $$ \Td\left( \prod_i S^{n_i}C\right)=\prod_i \td(\xi_i)^{n_i+1-g} e^{\theta_i \frac{\td \xi_i -\xi_i-1}{\xi_i}}.$$

\subsection{The first Chern class of $\mathcal{L}_K$}

\begin{theorem}
\label{th:c_1(L_K)}
Given a configuration $(K,g,d,\vec{n})$ and its associated vector of non-localized quasi-holes $\vec{p}=\vec{d}-(g-1)\vec{K}-K\vec{n}$, we have the following expression in $H^2(\prod S^{n_i}C\times \pic^d(C),\mathbb{Z})$:

    $$c_1(\mathcal{L}_K )=\sum_i (K_{ii}\theta_i -\eta_i + p_i \xi_i) + \sum_{i<j} K_{ij}\eta_{ij}$$

\end{theorem}

Let $s$ be the quotient map $s:C^{\sum n_i} \rightarrow \prod_i S^{n_i}C$ and $D$ be the divisor in $\prod S^{n_i}C$ whose pullback under $s$ is $\Delta_K$. Also, let $\mathcal{L}_0$ be the bundle $\mathcal{L}_K$ for $K$ the zero matrix. Then, by naturality of the first Chern class, we have the following relation in the cohomology of $C^{\sum n_i}\times \pic^d(C)$ (please excuse the abuse of notation):
$$s^*c_1(\mathcal{L}_K)=c_1(s^*\mathcal{L}_0)-\Delta_K=s^*c_1(\mathcal{L}_0)-s^*D$$
Since $s^*:H^2(\prod S^N(C),\mathbb{Z})\rightarrow H^2(C^N,\mathbb{Z})$ is an injection, the relation is also true in the cohomology of $\prod S^{n_i}C\times \pic^d(C)$:
\eqtag{c_1(\mathcal{L}_K)=c_1(\mathcal{L}_0)-D}{eq:relc1LkL0D}

\begin{lemma}\label{lemma:L_0}
    \eqtag{c_1(\mathcal{L}_0)=\sum_i (-\eta_i+d\xi_i)}{eq:c1L0}
\end{lemma}
\begin{proof}
    Consider the line bundle
$$\mathcal{L}^{\boxtimes\sum n_i}\longrightarrow S^{\sum n_i}C\times Pic^dC.$$ 

From article \cite{Klevtsov_Zvonkine_2025} we know that its first Chern class equals $-\eta+d\xi$ (see the main result for $b=0$). 

Now consider the forgetful map 
$$\Phi: \prod_i S^{n_i} C \times \pic^d C \to S^{\sum_i n_i} C \times \pic^d C $$ 

We have $c_1(\mathcal{L}_0)=\Phi^*c_1(\mathcal{L}^{\boxtimes\sum n_i})$. The fact that $\Phi^*(\eta)=\sum_i \eta_i$ is obtained by expanding
    $$\sum_{l=1}^g( \sum_{i} \alpha_i^l )\wedge \beta^l + \alpha^l\wedge ( \sum_{i} \beta_i^l)$$ 
We also have $\Phi^*(\xi)=\sum_{i}\xi_i$ because the inverse image is $\Phi^{-1}\xi=\bigcup \xi_i$ and each irreducible component $\xi_i$ has transversal multiplicity equal to 1. 
\end{proof}

We want to express $D$ in terms of $\theta_i,\xi_i$ and $\eta_{ij}$. In order to do so, we need to introduce a bit of notation first. Recall that $\Delta_{ij}$ is the divisor in $C^{\sum n_i}$ defined by $\Delta_{ij}:=\bigcup\{z^i_\lambda=z^j_\mu\}$ for $\lambda=1...n_i$ and $\mu=1...n_j$; with the additional condition that when $i=j$, we only take $\lambda < \mu$. We define $S\Delta_{ii}$ and $S\Delta_{ij}$ respectively as the images by $s$ of $\Delta_{ii}$ inside $S^{n_i}C$  and $\Delta_{ij}$ inside $S^{n_i}C\times S^{n_j}C$. In general, when the choice of $N$ is clear, we will denote $S\Delta$ the image of $C^N$'s diagonal inside $S^NC$ by the corresponding quotient map. Let's start by saying that, for $N>2g-1$, we have the following expression in $H^2(S^NC,\mathbb{Z})$:
\begin{equation}\label{expression of SDelta}
    S\Delta= 2(-\theta + (N+g-1)\xi)
\end{equation}

This is a particular case of a result presented in the book \cite[Ch. VIII, sec. 5]{Griffiths_book}, where the proof is carried out using the projection formula to compute the intersections of $S\Delta$ with all the classes in $A^\star(S^N C)$.

\begin{lemma}\label{transversal multiplicity}
Fix $i,j\in\{1,...,k\}$. Let $\Phi_{ij}: S^{n_i}C \times S^{n_j}C\rightarrow S^{n_i+n_j }C$ be the corresponding forgetful map.
    $$\Phi_{ij}^*S\Delta= S\Delta_{ii} + S\Delta_{jj}+ 2S\Delta_{ij}$$
\end{lemma}
\begin{proof}
    We will prove that the transversal multiplicity of $S\Delta_{ij}$ is $2$. Take neighborhoods $U$ of a point $\{z, z, z_1, z_2, ...\}\in S\Delta$ and $V$ of a fixed (general) antecedent of the form $(\{z, z_1, ...\},\{z, z_2, ...\})\in S\Delta_{ij}$. $\Phi_{ij}$ restrained to these neighborhoods is a two sheeted cover: If we take a different point $\{z, z', z_1, z_2, ...\}\in U$, it will have two antecedents in $V$, $(\{z, z_1, ...\},\{z', z_2, ...\})$ and $(\{z', z_1, ...\},\{z, z_2, ...\})$.
\end{proof}

\begin{lemma}
\label{lem:SDeltaiiijInNs}
For all $i,j\in\{1,...,k\}$,
    \begin{align}
    S\Delta_{ii}&=2(-\theta_i+(n_i+g-1)\xi_i) \label{eq:SDeltaiiInNS}\\
    S\Delta_{ij} &= -\eta_{ij} +n_j\xi_i + n_i\xi_j \label{eq:SDeltaijInNS}
\end{align}  
\end{lemma}

\begin{proof}
    The expression of $S\Delta_{ii}$ follows directly from the expression (\ref{expression of SDelta}). For the other one, we need a small computation:
    \begin{align*}
    S\Delta_{ij} &= \tfrac{1}{2}(\Phi_{ij}^*S\Delta- S\Delta_{ii}-S\Delta_{jj})\\
     &= \tfrac{1}{2}(-2(\theta_i+\theta_j+\eta_{ij}) + 2(n_i+n_j+g-1)(\xi_i+\xi_j) -S\Delta_{ii}-S\Delta_{jj})\\
     &= \tfrac{1}{2}(2(-\eta_{ij} +n_j\xi_i + n_i\xi_j)) \\
     &=-\eta_{ij} +n_j\xi_i + n_i\xi_j\\
\end{align*}
\end{proof}
 
\begin{lemma}\label{lemma:D}
    \eqtag{D=\sum_iK_{ii}(-\theta_i+(n_i+g-i)\xi_i)+\sum_{i<j}K_{ij}(n_j\xi_i+n_i\xi_j-\eta_{ij})}{eq:DinNS}
\end{lemma}
\begin{proof}
    Recall that we defined $D$ by $s^*D=\sum_{i\leq j} K_{ij}\Delta_{ij}$. By a transversal multiplicity computation similar to the one in lemma \ref{transversal multiplicity}, $s^*S\Delta_{ij}=\Delta_{ij}$ for $i<j$ and $s^*S\Delta_{ii}=2\Delta_{ii}$. One then has 
    $$D = \sum_{i<j} K_{ij}S\Delta_{ij} + \sum_i \tfrac{1}{2}K_{ii}S\Delta_{ii}$$
    We get the result by replacing $S\Delta_{ij}$ and $S\Delta_{ii}$ by the expression we obtained in Eq. \ref{eq:SDeltaiiInNS} and \ref{eq:SDeltaijInNS} of lemma \ref{lem:SDeltaiiijInNs}.
    
\end{proof}

Proof of theorem \ref{th:c_1(L_K)} follows by replacing in Eq. \ref{eq:relc1LkL0D} both $\mathcal{L}_0$ and $D$ in cohomology by the expression in Eq. \ref{eq:c1L0} and Eq. \ref{eq:DinNS} respectively.

\section{Applying the Grothendieck-Riemann-Roch theorem}
\label{sec:applyinggrr}

Recall
$$p: \prod_{i=1}^k S^{n_i} C \times \pic^d(C)  \to  \pic^d(C).$$

and the class we are interested in obtaining is $\ch(R^0p_*\mathcal{L}_K)$ as mentioned in subsection \ref{subsec:geometricmodel}. As stated in the outline of the paper, the Grothendieck–Riemann–Roch theorem applied to this proper map $p$ and the line bundle $\mathcal{L}_K$ gives:

\eqtag{\ch\left(\sum_i (-1)^iR^i p_\star \mathcal{L}_K)\right)=p_\star \left(e^{c_1(\mathcal{L}_K)} \td\left(\prod_{i=1}^k S^{n_i}C\right)\right).}{eq:GRR}

\subsection{Vanishing of the higher direct images through the Kodaira vanishing theorem}

\label{subsec:kodaira}

Let ${\mathcal{L}_K}_{|y}$ denote the restriction of the universal bundle ${\mathcal{L}_K}$ to the fiber of $p$ above $y \in \pic^d(C)$. We will find a sufficient condition so that the cohomology groups $H^i(\prod_i S^{n_i}C \times \pic^d(C)|y, {\mathcal{L}_K}_{|y})$ vanish for all $i>0, \, y\in \pic^d{C}$. In this case, by Grauert theorem, $R^ip_*\mathcal{L}_K=0$ for $i>0$, and $V_{K,d,g,\vec{n}}=R^0p_*\mathcal{L}_K$ is a locally free sheave (i.e. a vector bundle). In particular, the left hand side of Eq. \ref{eq:GRR} becomes $\ch\left(V_{K,d,g,\vec{n}}\right)$.

We will use the Kodaira vanishing theorem and prove the line bundles ${\mathcal{L}_K}_{|y}$ are ample for any $y$. Note that since we are on a smooth projective variety, the property of a line bundle being nef or ample depends only on its first Chern class (in cohomology).

\begin{lemma}
\label{lem:condition-K-diviseur-neff}

    Let $A\in M_k(\mathbb{Z})$ be a nonnegative (i.e. symmetric positive semi-definite)  matrix. Then the cohomology class $\nu=\sum_i A_{ii} \theta_i + \sum_{i<j} A_{ij} \eta_{ij}$ is nef on $\prod_i \pic^{n_i}C $ and so is its pullback to $\prod_i S^{n_i} C.$
\end{lemma}
\begin{proof}

    Take $\lbrace \mu_e\rbrace_{1\leq e \leq g k} $ a basis of holomorphic one forms on $\prod_{i=1}^{k} \pic^{n_i}C$ and let $M$ be a symmetric matrix of size $gk$. The (1,1)-form

    $$i \sum_{e,f} M_{ef} \mu_e \wedge \ov{\mu}_f.$$  
    is non-negative if and only if $M$ is. Furthermore, by \cite[Prop. 6.10]{Demailly_2012}, a cohomology class that can be represented by a non negative (1,1) form is nef . Our goal is thus to find a basis of holomorphic one forms on $\prod_i \pic^{n_i}C$ and show that in this basis $\nu=i \sum_{e,f} M_{ef} \mu_e \wedge \ov{\mu}_f$ with $M$ non negative.

    First, let's exhibit a particular basis of holomorphic one forms on $C$. Recall we introduced a canonical basis of cycles $a^1\dots b^g$ of $H_1(C,\mathbb{Z})$. By \cite[Cor. 2.2]{Mumford_Musili_2007} there is a unique basis $\omega^1,\dots,\omega^g$ of holomorphic one forms on $C$ which satisfy $\int_{a^r} \omega^s=\delta^{r s}$ and furthermore, the matrix $\tau$ defined by 
    $$\tau^{rs}:=\int_{b^r} \omega^s$$ is symmetric, and its imaginary part is definite positive. $\tau$ is called the period matrix of $C$. In terms of the real harmonic one forms $\alpha^1,\dots \beta^g$ dual to the cycles $a^1,\dots b^g$ introduced in section \ref{sec:symplectic-basis}, we have: \eq{\omega=\alpha+\tau \beta}
    where we $\omega$, $\alpha$, $\beta$ are column vectors containing the $\omega^r$ $\alpha^r$, $\beta^r$ for $1 \leq r \leq g$.

   We can now do the same construction layer by layer to obtain a basis on $gk$ holomorphic one forms on $\prod_i \pic^{n_i} C$. In section \ref{sec:symplectic-basis}, we introduced 1-forms $\lbrace \alpha_{i}^r,\beta_{i}^r,1\leq r\leq g \rbrace$ living on each layer $i$. For a given layer $i$, let 
    \eqtag{
    \omega_i:=\alpha_i+\tau \beta_i}{eq:omegai}
    where $\alpha_i$,$\beta_i$ are column vectors containing the $\alpha_i^r$,$\beta_i^r$, for $1\leq r \leq g$. The $g$ components $\omega_i^r$ of $\omega_i$ are a basis of holomorphic one forms on $\pic^{n_i} C$ since the holomorphic structure on $\Omega^1(\pic^{n_i}(C))$ (induced by the complex structure on $\pic^{n_i}(C)$) comes from the embedding $\mathbb{Z}^{2g} \to \mathbb{C}^g$ given by $\tau$ of the lattice defining $\pic^{n_i}(C)$.

    To express $\nu$ in terms of the $\omega_i^r,\ov{\omega_i^r}, 1\leq i \leq k,1\leq r \leq g$, we first use Eq. \ref{eq:ThetaAsDifferentialForm} and Eq. \ref{eq:EtaAsDifferentialForm} to obtain \eq{\nu &= \sum_i A_{ii} \theta_i + \sum_{i<j} A_{ij} \eta_{ij} \\
   &= \sum_{i,j} \sum_{r=1}^g A_{ij} \alpha^r_i \wedge \beta^r_j.}
   
   We then replace $\alpha_i^r,\beta_i^r$ by their expression in terms of $\omega_i^r,\ov{\omega_i}^r$ using equation \ref{eq:omegai} and it complex conjugate. Let $B=-i(\tau-\ov{\tau})^{-1}$. Because $\tau$ and $A$ are symmetric, we have:
    \eq{
        \nu &=\sum_{i,j} \sum_{r=1}^g A_{ij} \alpha^r_i \wedge \beta^r_j \\
        &=-\sum_{i,j} A_{ij} \left( ({\omega_i})^T B \ov{\tau} B \ov{\omega_j} + ({\ov{\omega_i}})^T B \tau B \omega_j\right) \\ 
        &=-\sum_{i,j} A_{ij} \left( ({\omega_i})^T B \ov{\tau} B \ov{\omega_j} + ({\ov{\omega_j}})^T B \tau B \omega_i\right) \\
        &=-\sum_{i,j} A_{ij} \left( ({\omega_i})^T B \ov{\tau} B \ov{\omega_j} - (\omega_i)^T B\tau  B \ov{\omega_j}\right) \\
        &=i \sum_{i,j} A_{ij}  ({\omega_i})^T B \ov{\omega_j}\\
    }

    From the last expression we see that for the basis $\lbrace \mu_e, 1\leq e \leq gk \rbrace $ constructed from the $\lbrace \omega_i^r, 1\leq i \leq k,1\leq r \leq g \rbrace$ one can write $\nu=i \sum_{e,f} M_{ef} \mu_e \wedge \ov{\mu_f}$ with $M=A\otimes B$. Since $A$ is non-negative and $B$ is positive definite, $M$ is non-negative.

\end{proof}

\begin{lemma}
\label{lem:pullbackofamplebyproj}
For $1\leq i \leq k$, let $L_i$ be ample line bundles on projective varieties $X_i$ . Then $L=\otimes_i \pi_i^\star L_i \to \prod_i X_i$ is ample, with $\pi_i$ the projection onto the i-th factor.
\end{lemma}

\begin{proof}

    Up to replacing $L$ (resp. the $L_i$) by powers of $L$  (resp. the $L_i$), we can suppose the $L_i$ are very ample (in the powers $L^n$, all the factors $L_i^n$ will be very ample for $n$ large enough). The sections of each $L_i$ then give an embedding $\prod_i X_i \to \prod_i \mathbb{P}^{N_i-1}$. Composing with the Segre embedding $\prod_i \mathbb{P}^{N_i-1} \to \mathbb{P}^{(\prod_i N_i)-1}$ gives an embedding $\prod_i X_i \to \mathbb{P}^{(\prod_i N_i)-1} $ that is given by sections of $L$.

\end{proof}

\begin{proposition} 

Let $(K,d,g,\vec{n})$ be a configuration such that $K-I$ is non-negative and the associated $\vec{p}$ satisfies $p_i>-(n_i+1-g)$ for all $1\leq i \leq k$.

Then for all $i>0$ and all $y\in \pic^d(C)$:
$$H^i\left(\prod_i S^{n_i}C,{\mathcal{L}_K}_{|y}\right)=0.$$ 
\label{prop:kodairavanishing}
\end{proposition}

\begin{proof}

Let $y \in \pic^d(C)$ and $L_K={\mathcal{L}_K}_{|y}$. Let $(K,d,g,\vec{n})$ be such a configuration. We write $\omega_X$ the canonical line bundle on $X=\prod_{i=1}^{k}S^{n_i }C$.
The result will follow from the Kodaira vanishing theorem if we show that $L_{K} \otimes {\omega_X}^{-1}$ is ample. From \cite{Klevtsov_Zvonkine_2025} we know that $\omega_{S^NC}=\theta-(N+1-g)\xi$, thus 
$$\omega_X=\sum_i\theta_i-(n_i+1-g)\xi_i.$$
By the results of section \ref{sec:computingthefirstChernclass}, $L_{K} \otimes \omega_X^{-1}$ has first Chern class

$$\sum_{i} (K_{ii}  \theta_i  + p_i \xi_i )+\sum_{i<j} K_{ij} \eta_{ij} - \sum_i \left(\theta_i -(n_i+1-g) \xi_i \right)$$ which can be rewritten

$$\sum_{i}(p_i+(n_i+1-g)) \xi_i + \sum_{i<j} (K-I)_{ij} \eta_{ij} + \sum_i (K-I)_{ii} \theta_i $$

Given that the cohomology class $\xi_i$ is ample over $S^{n_i} C$ (see \cite{Klevtsov_Zvonkine_2025}) and the assumption on $p_i$, each $(p_i+(n_i+1-g)) \xi_i$ is ample over $S^{n_i}C$. Then $\sum_{i} (p_i+(n_i+1-g))\xi_i$ is ample on $S^{n_1}C\times \dots\times  S^{n_k}C$ by lemma \ref{lem:pullbackofamplebyproj}

By lemma \ref{lem:condition-K-diviseur-neff}, the second two sums add up to a nef cohomology class. Thus all three sums add to an ample class by Kleiman’s criterion. By the Kodaira vanishing theorem, the cohomology groups $H^i(S^{n_1}C\times \dots \times S^{n_k}C,L_K)$ vanish for $i>0$.\end{proof}

\begin{remark}
    $K-I$ non negative implies in particular $K$ is positive.
\end{remark}

\begin{corollary}
\label{corr:kodaira}
    Let $(K,d,g,\vec{n})$ be a configuration as in proposition \ref{prop:kodairavanishing}. From the latter proposition as well as Grauert theorem, we get that $\sum_i (-1)^iR^i p_\star \mathcal{L}_K)=R^0 p_\star \mathcal{L}_K :=V_{K,g,d,\vec{n}} $ is a vector bundle and that the l.h.s. of the Grothendieck-Riemann-Roch formula compute its Chern character: $\ch\left(\sum_i (-1)^iR^i p_\star \mathcal{L}_K)\right)=\ch(V_{K,g,d,\vec{n}})$.
\end{corollary}
\color{black}

\subsection{The case when there are no quasi-holes}
\label{subsec:computationoftherhs}

Now that we know a necessary condition on a configuration $(K,d,g,\vec{n})$ to have the equality $\ch\left(V_{K,g,d,\vec{n}}\right)=\ch\left(\sum_i (-1)^iR^i p_\star \mathcal{L}_K)\right)$, we will focus on the computation of $$ p_\star \left(e^{c_1(\mathcal{L}_K)} \td\left(\prod_{i=1}^k S^{n_i}C\right)\right).$$

We will carry out the computation in two steps using the fact that the projection map $p$ on the $\pic^d(C)$ factorizes as:
$$p: \prod_{i=1}^k S^{n_i} C \times \pic^d(C) \overset{p_1}{\to} \prod_{i=1}^k \pic^{n_i} C \times \pic^d(C) \overset{p_2}{\to} \pic^d(C)$$ where $p_1$ is the map which, for each $i$, associates to each configuration of points in $S^{n_i}C$ its corresponding divisor in $\pic^{n_i}C$ and $p_2$ is the projection onto the last factor. We first need the few following lemmas:

\begin{lemma}
\label{lem:coeffextractiontd}

Recall that $\td(x)=\frac{x}{1-e^{-x}}$.  Let $p$, $r$,$a$ be integers, with $a$ non negative.  We have
$$[x^r] \td(x)^{r+1} e^{px} \left(\frac{\td(x)}{x}-1\right)^a= \binom{r+p}{p-a}$$

where $[x^r]$ denote the extraction of the coefficient of $x^r$ and the binomial coefficient is taken with the convention that it is zero when the lower entry is negative or the upper entry is non-positive. See \cite{Klevtsov_Zvonkine_2025} for the proof, which is carried out using the residue theorem.
\end{lemma}

\begin{lemma}
\label{lem:firstpushforward_lemma1}
Let $n_i,p_i$ be integers. For $1\leq i\leq k$ fixed, 
$${p_1}_\star \left( e^{p_i \xi_i}\td(\xi_i)^{n_i+1-g}  e^{\theta_i \frac{\td \xi_i -\xi_i-1}{\xi_i}} \right)=f_i(\theta_i)$$

where $f_i$ is the polynomial defined as $$f_i(x)= \sum_{a\geq 0} \frac{1}{a!} \binom{n_i-g+p_i}{p_i-a} x^a$$
\end{lemma}

\begin{proof}
   
Let $B(x)=\frac{\td(x)-1-x}{x}$ and $A_i(x)=e^{p_i x} \td(x)^{n_i+1-g}$. Developing the exponential, we get  
$$e^{p_i \xi_i} \td(\xi_i)^{n_i+1-g}  e^{\theta_i \frac{\td \xi_i -\xi_i-1}{\xi_i}}=\sum_k \frac{1}{k!} \theta_i^k A_i(\xi_i) B^k(\xi_i)$$
 
For any $a \in \mathbb{N}$,
\eq{
    [\theta_i^a] \,  {p_1}_\star \left(\sum_k \frac{1}{k!} \theta_i^k A_i(\xi_i) B^k(\xi_i)\right)
    &=[\theta_i^a] \, \sum_k \frac{1}{k!}\theta_i^k {p_1}_\star(\xi_i^{n_i-g+a-k}) [x^{n_i-g+a-k}] A_i(x) B^k(x)\\
    &=\frac{1}{a!} \sum_k  \binom{a}{k} [x^{n_i-g+a-k}] A_i(x) B^k(x)\\
    &=\frac{1}{a!} [x^{n_i-g}] A_i(x) \left(B(x)+\frac{1}{x}\right)^a\\
    &=\frac{1}{a!} [x^{n_i-g}] e^{p_i x} \td(x)^{n_i+1-g} \left(\frac{\td(x)}{x}-1\right)^a\\
}
    The result follows from lemma \ref{lem:coeffextractiontd}.

\end{proof}

\begin{lemma}
Let $(K,d,g,\vec{n})$ be a configuration. The complete expression for the pushforward by $p_1$ is:
\label{lem:firstpushforward}
    $${p_1}_\star \left(e^{c_1(\mathcal{L}_K)} \td\left(\prod_{i=1}^k S^{n_i}C\right)\right)=\left( \prod_i f_i(\theta_i) \right) e^{\sum_i (K_{ii}\theta_i -\eta_i) + \sum_{i<j} K_{ij}\eta_{ij} }$$
    
\end{lemma}

\begin{proof}

Recall that in terms of the cohomology classes introduced in section \ref{sec:importantcohomologyclasses}, $\td(S^{n_i}C)$ reads

$$\td(S^{n_i}C)=\td(\xi_i)^{n_i+1-g} e^{\theta_i \frac{\td \xi_i -\xi_i-1}{\xi_i}}.$$ 

Replacing $c_1(\mathcal{L}_K)$ by its expression obtained in theorem \ref{th:c_1(L_K)}, we get

\eq{
e^{c_1(\mathcal{L}_K)} \td\left(\prod_{i=1}^k S^{n_i}C\right)&=e^{\sum_i (K_{ii}\theta_i -\eta_i) + \sum_{i<j} K_{ij}\eta_{ij} } \prod_i e^{p_i x} \td(\xi_i)^{n_i+1-g}  e^{\theta_i \frac{\td \xi_i -\xi_i-1}{\xi_i}}
}

Since the projection $p_1$ acts on each factor independently, we get ${p_1}_\star \left(  \prod_i \xi_i^{n_i-g+k} \right)=\prod_i \frac{\theta_i^k}{k!}$.

Furthermore, since all the other cohomology classes are defined via pullbacks from the $\pic^{n_i}$ to $S^{n_i} C$, the projection formula implies that the pushforward on any monomial in all the $\theta_i, \eta_i,\eta_{ij}$ and some $\xi_i$ to the power ${n_i-g+k}$ is simply the same monomial where we replaced $\xi_i^{n_i-g+k}$ by $\frac{\theta_i^k}{k!}$. The result then follows from lemma \ref{lem:firstpushforward_lemma1}.

\end{proof}

We will now compute the pushforward by $p$ in the specific case where the particles numbers are exactly given by $K\vec{n}=\vec{d}-(g-1)\vec{K}$.

\begin{proposition}
\label{prop:tomainthm}
    Let $(K,g,d,\vec{n_0})$ be a configuration such that $K\vec{n}_0=\vec{d}-(g-1)\vec{K}$. Then

    $$p_\star \left(e^{c_1(\mathcal{L}_K)} \td\left(\prod_{i=1}^k S^{n_i}C\right)\right)=\det(K)^g e^{|K^{-1}|\theta}$$

\end{proposition}

\begin{proof}
Let $(K,g,d,\vec{n_0})$ such a configuration. The associated $\vec{p}$ is $\vec{p}=\vec{d}-K\vec{n}-(g-1)\vec{K}=0$ which implies the $f_i$' appearing in \ref{lem:firstpushforward} are all one. Thus

$${p_2}_\star \left( \left[\prod_i f_i(\theta_i) \right] e^{\sum_i (K_{ii}\theta_i -\eta_i) + \sum_{i<j} K_{ij}\eta_{ij} } \right)={p_2}_\star\left(e^{\sum_{i<j} K_{ij} \eta_{ij} + \sum_i (K_{ii} \theta_i -\eta_i) } \right).$$

Recall that the classes $\eta_{ij}$, $\theta_i$, $\eta_i$ can be expressed in terms of the $\alpha_i^r$, $\beta_j^r$ the symplectic bases introduced in section \ref{sec:symplectic-basis} of the spaces $H^1(\pic^{n_i},\mathbb{C})$ and in terms of $\alpha^r$, $\beta^r$ which form a symplectic basis on $\pic^d(C)$. The map ${p_2}_\star$ selects the terms that contain the top form $\prod_{i=1}^k \alpha_i^1 \beta_i^1 \dots \alpha_i^g \beta_i^g$ and integrate this factor out, leaving a cohomology class on $\pic^d(C)$.\\

For the extraction to be more convenient, we introduce $\psi^r=(\beta_{1}^r,\dots, \beta_{k}^r)^T$, $\ov{\psi^r}=(\alpha_{1}^r,\dots, \alpha_{k}^r)^T$, $\lambda^r=(\beta^r,\dots,\beta^r)^T$ and $\ov{\lambda^r}=(\alpha^r,\dots,\alpha^r)^T$ for $1\leq r \leq g$. To shorten notations we omit the $\wedge$ symbol when multiplying forms.

One can then write $$\sum_i (K_{ii}\theta_i -\eta_i) + \sum_{i<j} K_{ij}\eta_{ij}=\sum_{r=1}^g \left( (\ov{\psi^r})^T K \psi^r - \ov{\lambda^r} \psi^r - \ov{\psi^r} \lambda \right)$$
and the computation of the pushforward by $p_2$ amounts to extracting the terms containing $\prod_{i=1}^{k} \prod_{r=1}^g \ov{\psi^r}_i \psi^r_i $ in the expression 
\eqtag{e^{\sum_i (K_{ii}\theta_i -\eta_i ) + \sum_{i<j} K_{ij}\eta_{ij} }=\prod_{r=1}^g \left( e^{(\ov{\psi^r})^T K \psi^r - \ov{\lambda^r} \psi^r - \ov{\psi^r} \lambda} \right).}{eq:secondpushforward}

This is a purely combinatorial problem, we can forget that the $\psi_i^r$,$\ov{\psi^s_j}$ are differential forms and simply see them as generators of a so called Grassmann algebra. For more details about Grassmann algebras, see \cite{Berezin_1987,Caracciolo_Sokal_Sportiello_2013}. Here, we recall that for $R$ a commutative ring with identity and m symbols $\chi_a, 1 \leq a \leq m$, the associated Grassmann algebra is the ring of polynomials in the $\chi_a$ quotiented by the relations $\chi_a \chi_b + \chi_b \chi_a = 0$ for $1\leq a < b \leq m$ and $\chi_a^2=0, 1\leq a \leq m$. On such an algebra we can define an $R$-linear operator $\int d \chi_a$ as 
 $$\int d\chi_a \chi_{a_1}\dots \chi_{a_q}=
\begin{cases}
    (-1)^{\delta-1} \chi_{a_1}\dots \chi_{a_{\delta-1}},\chi_{a_{\delta+1}}\dots, \chi_{a_q} &\text{if } a = a_\delta \\
    0 & \text{if } a \notin \lbrace a_1,\dots, a_q\rbrace.
\end{cases}
$$ for any monomial $\chi_{a_1}\dots \chi_{a_q}$ with $a_1 <\dots < a_q$. A Grassmann algebra is a graded algebra, with $\deg\left(\chi_{a_1}\dots \chi_{a_q}\right)=q$.

In our case, we take the Grassmann algebra associated to the ring $R=\mathbb{R}$ and the ordered symbols $$\psi_1^1,\ov{\psi_1^1},\psi_1^2,\ov{\psi_1^2},\dots, \psi_1^g,\ov{\psi_1^g},\psi_2^1,\ov{\psi_2^1},\dots \dots, \psi_k^g,\ov{\psi_k^g},\lambda_1^1,\ov{\lambda_1^1},\dots,\lambda_1^g,\ov{\lambda_1^g}.$$ Note that this choice of ordering comes from the orientation on $\prod_i \pic^{n_i} C \times \pic^d(C)$. 
As an example, fix $i\in[k],r\in[g]$ and $\kappa$ an element of the Grassman algebra which doesn't contain $\psi_i^r$. Then for any $j \in[k],s\in[g]$ we have $$\int d{\psi^r_i} \psi^s_{j} \kappa  = \kappa \delta_i^j \delta_r^s$$ while $$\int d{\psi^r_i} \kappa \psi^s_{j}   = (-1)^{\deg(\kappa)} \kappa \delta_i^j \delta_r^s$$
where $\delta_i^j$ is the Kronecker symbol.

Note that in $\prod_{i=1}^{k} \prod_{r=1}^g \ov{\psi^r_i} \psi^r_i $, the ordering in $r$ and $i$ does not matter because each elements of the form $\ov{\psi^r_i} {\psi^r}_i$ is even in the Grassmann algebra.

In our computation of the pushforward under $p_1$, we regrouped terms when they belonged to the same layer $S^{n_i} C$, for $1\leq i \leq k$. 
Yet, in the notations we introduced, the quantity we want to pushforward takes the form of a product over upper indices $1\leq r \leq g$, so we will carry out the extraction as 
$${p_2}_\star= \int \prod_{r=1}^g \left[ D(\psi^r,\ov{\psi^r}) \right].$$ 

where $D(\psi^r,\ov{\psi^r})= \prod_{i=1}^k d\psi^r_i d\ov{\psi^r}_i$.

The analogue of Fubini's theorem for Berezin integrals then gives:

\eq{
{p_2}_\star \left(\prod_{r=1}^g \left( e^{(\ov{\psi^r})^T K \psi^r - \ov{\lambda^r} \psi^r - \ov{\psi^r} \lambda} \right) \right)&=\int   \prod_{r=1}^g \left[ D(\psi^r,\ov{\psi^r}) \right] \prod_{r=1}^g e^{(\ov{\psi^r})^T K \psi^r - \ov{\lambda^r} \psi^r - \ov{\psi^r} \lambda}\\
&=\prod_{r=1}^g \int   D(\psi^r,\ov{\psi^r})   e^{(\ov{\psi^r})^T K \psi^r - \ov{\lambda^r} \psi^r - \ov{\psi^r} \lambda}
} 
By Wick’s theorem for "complex fermions" (see \cite{Caracciolo_Sokal_Sportiello_2013}[Thm. A.16]), each term in the product can be expressed as   
$$\int   D(\psi^r,\ov{\psi^r})   e^{(\ov{\psi^r})^T K \psi^r - \ov{\lambda^r} \psi^r - \ov{\psi^r} \lambda}=\det(K) e^{-\ov{\lambda^r} K^{-1} \lambda^r}$$

Recalling that all the components of $\ov{\lambda^r}$ (resp. $\lambda^r$) are equal to $\alpha^r$ resp. $\beta^r$ and that  $\theta=\sum_r \alpha^r \beta^r$, we get $\ov{\lambda^r} K^{-1} \lambda^r=|K^{-1}| \alpha^r \beta^r$ and thus

    $$\prod_{r=1}^g \det(K) e^{-\ov{\lambda^r} K^{-1} \lambda^r}=\det(K)^g e^{-|K^{-1}| \theta}$$

\end{proof}

Gathering what we have shown so far, we can prove theorem \ref{thm:main} and \ref{thm:maximalparticules}.

\begin{proof}[Proof of theorem 1]

Let $(K,d,g,\vec{n}_0)$ be a configuration such that $K-I$ is non negative and $K\vec{n}_0+(g-1)\vec{K}=\vec{d}$ .

By the the Grothendieck-Riemann-Roch formula, as well as corollary \ref{corr:kodaira} and proposition \ref{prop:tomainthm}, we have 
$$\ch\left( V_{K,d,g,\vec{n}_0}\right)=p_\star \left(e^{c_1(\mathcal{L}_K)} \td\left(\prod_{i=1}^k S^{n_i}C\right)\right)=\det(K)^g e^{|K^{-1}|\theta}$$

\end{proof}

\begin{proof}[Proof of theorem 2]

Let $(K,d,g,\vec{n})$ be a configuration such that the associated $\vec{p}$ has a negative component $p_i$. Let $y \in \pic^d(C)$ and $L_K={\mathcal{L}_K}_{|y}$. Recall that a fiber of $S^{n_i} C \to \pic^{n_i} C$ is isomorphic to $\mathbb{P}^{n-g}$. Looking at the Chern class of $L_K$, one can see that the restriction of $L_K$ to such a fiber is isomorphic to $O(p_i)$. $L_K$ thus cannot have global sections, otherwise by restriction one would get global sections of $O(p_i)$. Since this is true for any $y\in \pic^d(C)$, $\rk(V_{k,d,g,\vec{n}})=0$.\\

For the second assertion of the theorem, take a configuration $(K,d,g,\vec{n_0})$ such that $K-I$ is non negative and $K\vec{n}_0+(g-1)\vec{K}=\vec{d}$. We already know by theorem \ref{thm:main} that $\rk(V_{k,d,g,\vec{n}_0})>0$.

Let $\vec{n}$ such that $\vec{n}>\vec{n}_0$ in the sense that $\vec{n}-\vec{n}_0$ has non negative components with at least one being positive. Then $\vec{p}=\vec{d}-K\vec{n}-(g-1)\vec{K}$ has a negative component, and thus $\rk(V_{K,d,g,\vec{n}})=0$.\\

For the last assertion of the theorem, take a configuration $(K,d,g,\vec{n_0})$ such that $K-I$ is non negative, $K\vec{n}_0+(g-1)\vec{K}=\vec{d}$ and $C_i:=\sum_{j} K_{ji}^{-1} \geq 0$ for $1\leq i \leq k$. For any other particle number configuration $\vec{n}$, one has $\vec{n}-\vec{n}_0=-K^{-1} \vec{p}$. Multiplying by $\vec{1}=(1,\dots,1)^T$ this relation, we get $N-N_0=-\vec{1}^TK^{-1} \vec{p}=-(C_1,\dots,C_k)^T \vec{p}$ where $N$ (resp. $N_0$) are the total particle numbers for the particle configuration $\vec{n}$ (resp. $\vec{n_0}$). In order for quantum states with particle number $\vec{n}$ to exist, we need the components of $\vec{p}$ to be non negative, in which case $-(C_1,\dots,C_k)^T \vec{p}\leq 0$

\end{proof}

\color{black}

\subsection{The general case}

For $K \in M^k(\mathbb{R})$ and $I \subset [k]$, we denote by $K^\sharp$ its adjugate matrix, i.e. the transpose of the cofactor matrix of $K$ and by $K_I$ the submatrix of the matrix $K$ obtained by keeping only the columns and rows corresponding to the indices in $I$. When $I \subset [k]$, we denote by $I^c$ the complementary of $I$ in $[k]$. 

\begin{proposition}
\label{prop:multilayer_Chern_character}
For any $K,g,d$ and $\vec{n}$ with $K$ symmetric,
    \eq{
    p_\star \left(e^{c_1(\mathcal{L}_K)} \td\left(\prod_{i=1}^k S^{n_i}C\right)\right)=  \sum_{\substack{|v|+|w|=g}} C_{v,w}  \prod_{I\subset{k}} |{K_{I^c}}^{\sharp}|^{v_I} \det(K_{I^c})^{w_I}\frac{(-\theta^{|v|})}{|v|!}}

where $$C_{v,w}= \binom{|v|}{v}\binom{g-|v|}{w} \prod_{i=1}^k \binom{n_i-g+p_i}{p_i-\sum_{ I | i \in I} (v_I+w_I)}.$$

In these expressions, $v=(v_I)_{I\subset [k]}$ and $w=(w_I)_{I\subset [k]}$ are collections of $2^k$ non negative integers indexed by $I$. For such collections  we write $|v|$ (resp. $|w|$) their sum.

$\sum_{\substack{|v|+|w|=g}}$ is the sum over all possibles values of the $v_I$'s and $w_I$'s such that $$|v|+|w|=\sum_I v_I + w_I=g.$$ 
As $v$ is a multi-indice, $\binom{|v|}{v}$ for instance is a multinomial coefficient: $\binom{|v|}{v}=\frac{|v|!}{\prod_{I\subset [k]} v_I!}$.

\end{proposition}

With the notations introduced in the proof of proposition \ref{prop:tomainthm}, we can express the pushforward by $p$ as:
\eqtag{
\int \prod_{r=1}^g  D(\psi^r,\ov{\psi^r}) \left( \prod_{i=1}^k f_i \left(\sum_{r=1}^g \ov{\psi}_i^r \psi_i^r \right)  e^{(\ov{\psi^r})^T K \psi^r - \ov{\lambda} \psi^r - \ov{\psi^r} \lambda} \right)\\
}{eq:full_computation}

The difference with the case of theorem \ref{thm:main} is that since now the $p_i$'s are not all zeros, all the $f_i's$ are not equal to one. We will thus have monomials in front of $$\prod_{r=1}^g e^{(\ov{\psi^r})^T K \psi^r - \ov{\lambda^r} \psi^r - \ov{\psi^r} \lambda} $$ coming from the expansion of $\prod_{i=1}^k f_i \left(\sum_{r=1}^g \ov{\psi}_i^r \psi_i^r \right)$. Those monomials are of the form $\prod_r (\ov{\psi^r} \psi^r)_{I_r}$ for some given $I_r \subset [k]=\lbrace 1,\dots, k\rbrace$ where for each $r$, $(\ov{\psi^r} \psi^r)_{I_r}:=\prod_{i\in I_r}\ov{\psi^r}_i \psi^r_i$. In all these products the ordering again does not matter since each pair $\ov{\psi^r}_i \psi^r_i$ is even. 

We will first carry out the computation for a single monomial, before proving \ref{prop:multilayer_Chern_character}:

\begin{lemma}
\label{lemma:intermediate-noquasiholes}
   Fix $I_1 \subset [k] , \dots, I_g \subset [k]$. We have 
   $$\int    \prod_{r=1}^g \left[ D(\psi^r,\ov{\psi^r}) \right] \prod_{r=1}^g \left[  (\ov{\psi^r} \psi^r )_{I_r} e^{(\ov{\psi^r})^T K \psi^r - \ov{\lambda^r} \psi^r - \ov{\psi^r} \lambda} \right]=\sum_{F\subset[g]} (-\alpha \beta)^F \prod_{r \in F} |{K_{I^c_r}}^{-1}| \prod_{r=1}^g\det(K_{I_r^c})$$ where we wrote $(\alpha \beta)^F=\prod_{r\in F} \alpha^r \beta^r $
\end{lemma}

\begin{proof}
    
Fix $I_1 \subset [k] , \dots, I_g \subset [k]$. For all $1\leq r\leq g$, Wick's theorem in this context implies

$$\int D(\psi^r,\ov{\psi^r}) (\ov{\psi^r} \psi^r )_{I_r} e^{(\ov{\psi^r})^T K \psi^r - \ov{\lambda^r} \psi^r - \ov{\psi^r} \lambda}=\det(K_{I_r^c}) e^{-|{K_{I_r^c}}^{-1}|\alpha^r \beta^r}$$

thus

\eq{
&\int    \prod_{r=1}^g \left[ D(\psi^r,\ov{\psi^r}) \right] \prod_{r=1}^g \left[  (\ov{\psi^r} \psi^r )_{I_r} e^{(\ov{\psi^r})^T K \psi^r - \ov{\lambda^r} \psi^r - \ov{\psi^r} \lambda} \right]\\&=
\prod_{r=1}^g \int D(\psi^r,\ov{\psi^r}) (\ov{\psi^r} \psi^r )_{I_r} e^{(\ov{\psi^r})^T K \psi^r - \ov{\lambda^r} \psi^r - \ov{\psi^r} \lambda} \\
&=\prod_r \det(K_{I_r^c}) e^{-|{K_{I_r^c}}^{-1}|\alpha^r \beta^r}
}

Developing the product of exponential, we obtain 

$$\prod_r \det(K_{I_r^c}) e^{-|{K_{I_r^c}}^{-1}|\alpha^r \beta^r}=\sum_{F\subset[g]} (-\alpha \beta)^F \prod_{r \in F} |{K_{I^c_r}}^{-1}| \prod_{r=1}^g\det(K_{I_r^c}) $$

\end{proof}

\begin{proof}[Proof of proposition \ref{prop:multilayer_Chern_character}]

In lemma \ref{lemma:intermediate-noquasiholes} we computed the pushforward for all monomials. We start from Eq. \ref{eq:full_computation} and expand $\prod_{i=1}^k f_i \left(\sum_{r=1}^g \ov{\psi}_i^r \psi_i^r \right)$. The expansion reads 
$$\prod_{i=1}^k f_i \left(\sum_{r=1}^g \ov{\psi}_i^r \psi_i^r \right)= \prod_i {a_i!}  \sum_{\substack{I_1,\dots,I_g\\ \sharp\lbrace \bullet, c\in I_\bullet\rbrace=a_c}} \prod_{r} (\ov{\psi^r} \psi^r)_{I_r}$$ where $\sharp\lbrace \bullet, c\in I_\bullet\rbrace $ is the number of sets among $I_1,\dots I_g$ which contains $c$. Inserting this expansion in Eq. \ref{eq:full_computation}, we get

\eq{
&\sum_{a_1,\dots,a_k} \left[\prod_{i=1}^k \binom{n_i-g+p_i}{p_i-a_i}\right]  \sum_{\substack{I_1,\dots,I_g\\  \sharp\lbrace \bullet, c\in I_\bullet\rbrace=a_c}}  \sum_{F\subset[g]} (-\alpha \beta)^F \prod_{r \in F} |{K_{I^c_r}}^{-1}| \prod_{r=1}^g\det(K_{I_r^c})   \\
= &\sum_{a_1,\dots,a_k} \left[\prod_{i=1}^k \binom{n_i-g+p_i}{p_i-a_i}\right]  \sum_{\substack{I_1,\dots,I_g\\  \sharp\lbrace \bullet, c\in I_\bullet\rbrace=a_c}} \sum_m \frac{(-\theta)^m}{m!} \prod_{r=1}^m |{K_{I^c_r}}^{-1}| \prod_{r=1}^g\det(K_{I_r^c}) 
}

We get the final expression by regrouping the $|K^{-1}_I|$ as well as the $\det(K_I)$ when they are indexed by the same set $I$.

\end{proof}

By the Grothendieck-Riemann-Roch theorem and the result of subsection \ref{subsec:kodaira} about the vanishing of higher direct images, we get:

\begin{theorem}
\label{thm:multilayer_Chern_character}
Let $K,g,d$ and $\vec{n}$ with $K-I$ non-negative and $p_i>-(n_i+1-g)$ for all $i$. Then
    \eq{
    \ch\left(V_{K,g,d,\vec{n}}\right)=  \sum_{\substack{|v|+|w|=g}} C_{v,w}  \prod_{I\subset{k}} |{K_{I^c}}^{\sharp}|^{v_I} \det(K_{I^c})^{w_I}\frac{(-\theta^{|v|})}{|v|!}
    }
\end{theorem}

\paragraph{Some examples\\}

Before looking at the asymptotics for large magnetic fields, we give a few examples. Let $(K,d,g,\vec{n})$ be a configuration with 

\[
K=\begin{pmatrix}
b+1 & b & \cdots & b \\
b & b+1 & \cdots & b \\
\vdots & \vdots & \ddots & \vdots \\
b & b & \cdots & b+1
\end{pmatrix}
\]
for $b$ some positive integer. One has $K-I\geq 0$, the associated degeneracy is $\det(K)^g=(kb+1)^g$ and the conductivity is $|K^{-1}|=\frac{k}{kb+1}$. Note that this matrix satisfies $C_i=\sum_{j} K^{-1}_{ij}=\frac{1}{kb+1}>0$ for all $1\leq i \leq k$, hence the configuration maximizing the total number of particle is the configuration without non-localized quasi-holes.

Another example, which won't satisfy this condition is $$K=\begin{pmatrix}
10 & 3 \\
3 & 2  \\
\end{pmatrix}.$$

One has $K-I\geq 0$, $K^{-1}=\frac{1}{11} 
\begin{pmatrix}
2 & -3 \\
-3 & 10  \\
\end{pmatrix}$. $|K^{-1}|=\frac{14}{11}$. Note that $C_2=\sum_{j} K^{-1}_{2j}=\frac{7}{11}$ but $C_1=\sum_{j} K^{-1}_{1j}=-\frac{1}{11}\leq 0$. Let $(K,g,d,\vec{n_0})$ be a configuration such that $\vec{p_0}=0$. Given any other configuration $(K,g,d,\vec{n})$ where only the number of particles in each layer is different, recall that we have $\vec{n}=\vec{n_0}-K^{-1}\vec{p}$. In particular, $N=N_0-\sum C_ip_i$. This is an example of a configuration where $\vec{n_0}$ does not maximize the total number of particles. We can obtain multiple other configurations with a higher total number of particles from $(K,g,d,\vec{n_0})$ by raising the number of quasi-holes provided that $p_1>7p_2$ and that both are a multiple of $11$. For example: if one takes $p_1=88$ and $p_2=11$, the number of particles in layer $1$ decreases by $13$ but the number of particles in layer $2$ increases by $14$ compared to the situation with $p_1=p_2=0$. Now let's take $p_1=11$ and $p_2=0$. Compared to the situation with $p_1=p_2=0$, the number of particles in layer $1$ decreases by $2$ and the number of particles in layer $2$ increases by $3$ . Let $m=\left \lfloor{\frac{(n_0)_1}{2}}\right \rfloor$. The maximum total number of particles is $N=N_0+m$, it corresponds to $\vec{n}=\vec{n_0}+(-2m, 3m)$. It amounts to emptying the first layer.

\section{Asymptotic for large magnetic field}
\label{sec:asymptotics}

The formula obtained in the general case is rather complicated. However, we can study its asymptotics for large magnetic field and large particle numbers. 

\begin{lemma}
    \label{lem:Lidet}

    Let $K$ be an invertible symmetric matrix of size $k$, and write $C_i=\sum_{j=1}^k K^{-1}_{ij}$. For any  $i \in [k]$, one has:
        \eqtag{\frac{\det(K_{\lbrace i \rbrace^c })}{\det(K)} \left(  |K^{-1}| - |K^{-1}_{\lbrace i \rbrace^c }| \right)=C_i^2}{eq:detLi}
        where we recall that the notation $K_{\lbrace i \rbrace}^c$ means the matrix obtained from $K$ by keeping the rows and lines in the set $\lbrace i \rbrace^c = \lbrace 1,\dots,i-1,i+1,\dots k$.
    
    \end{lemma}
    
    \begin{proof}
    
    Take the Grassmann algebra generated by $\psi_1,\dots,\psi_k$, $\ov{\psi}_1,\dots, \ov{\psi}_k$,$\alpha$,$\beta$ and let $\lambda=(\beta,\dots,\beta)$,  $\ov{\lambda}=(\alpha,\dots,\alpha)$.
    
    Let $I \subset [k]$ and $$P=\int D(\psi,\ov{\psi}) (\ov{\psi} \psi )_{I} e^{\ov{\psi} K \psi - \ov{\lambda} \psi - \ov{\psi} \lambda}$$
    
    On the one hand,  $P=\det(K_{I^c}) e^{-|{K_{I^c}}^{-1}|\alpha \beta}$. On the other hand:
    
    \eq{
        P &= e^{-\ov{\lambda} K^{-1} \lambda}\int D\psi \ov{\psi}(\psi+K^{-1}\lambda)_I( \ov{\psi}+K^{-1}\ov{\lambda})_I  e^{\ov{\psi^T} K \psi} \\
        &=e^{-|K^{-1}|\alpha \beta} \left( \det(K_{{I}^c})+\alpha \beta \sum_{i,j \in I} C_i C_j (-1)^{i+j}\det{K}_{({I-i,I-j})^c} \right)
    }
    
    Taking $I=\lbrace i \rbrace $, and comparing coefficients in front of $\alpha \beta$, we obtain relation \ref{eq:detLi}.

    \end{proof}

\begin{proposition}
    Let $(K,g,d,\vec{n})$ be a configuration such that $C_i=\sum_j K_{ji}^{-1} > 0$ for all $i$. We study the asymptotics for $K,g$ fixed and large magnetic field, i.e. $d\to \infty$ with the assumption that the system is connected to a reservoir of particles, so that the total number of particle is maximized. Then $$\vec{n} = K^{-1} \vec{d} + o(d).$$
\end{proposition}

\begin{proof}

Let $(K,d,g,\vec{n})$ as in the proposition. Increasing $d$, the solution $\vec{n_0}$ to $K\vec{n}_0=\vec{d}-(g-1)\vec{K}$ does not necessarily have integer coefficients. Since we assume the $C_i$ to be non-negative, this $n_0$ maximizes the total particle number $\vec{1}^T \vec{n_0}$. Let us write $\vec{n_0}'$ the closest integer vector that maximizes $\vec{1}^T \vec{n}_0'$. We write $\vec{p}=\vec{d}-K\vec{n}_0'-(g-1)\vec{K}$. This closest vector satisfies $||\vec{n}_0'-\vec{n}_0||_\infty \leq 1$ and thus $||p||_{\infty}\leq \max_j \left((\sum_i K_{ij},1\leq j \leq k)\right)$. We get the result taking the asymptotics of the relation $\vec{p}=\vec{d}-K\vec{n}_0'-(g-1)\vec{K}$.

\end{proof}

\begin{proposition}
Let $(K,g,d,\vec{n})$ be a configuration such that $C_i=\sum_j K_{ji}^{-1} > 0$ for all $i$.
We study the asymptotics for $K,g$ fixed and large magnetic field, with fixed quasiholes numbers $\vec{p}$. We chose a sequence of stricklty increasing degrees $d$ such that the solution $\vec{n_0}$ to $K\vec{n}_0=\vec{d}-\vec{p} -(g-1)\vec{K}$ have integer coefficients. Then for this sequence we have 

$$\sigma=\left(|K^{-1}| - \sum_i \frac{p_i}{n_i} C_i^2 + o\left(\frac{1}{d}\right) \right)(-\theta)$$
\end{proposition}

\begin{proof}

Let $(K,d,g,\vec{n})$ as in the proposition. 
    Taking the asymptotic in $d$ of the expression of theorem \ref{thm:multilayer_Chern_character}, one obtains:
$$\sigma=\left(|K^{-1}| + \sum_i \frac{p_i}{n_i} \frac{\det(K_{\lbrace i \rbrace^c })}{\det(K)} \left(  |K^{-1}_{\lbrace i \rbrace^c }| - |K^{-1}| \right) + o\left(\frac{1}{d}\right) \right)(-\theta)$$

One obtains the result using lemma \ref{lem:Lidet}.
\end{proof}

\printbibliography

\end{document}